\documentclass[10pt,  a4paper]{article}
\usepackage{times}
\usepackage{dsfont}
\usepackage{fullpage}
\usepackage{titlesec}

\newcommand{\prodlambda}{\prod_{\lambda}(\{T_1, \ldots, T_i\})}
\newcommand{\prodk}{\prod_{k}(T_1, \ldots, T_i)}
\newcommand{\ind}[1]{\mathds{1}_{\{#1\}}} 

\newcommand{\E}[1]{\mathbb{E}\left[#1\right]} 
\newcommand{\Pro}[1]{\mathbb{P}\left(#1\right)}
\newcommand{\st}{\ensuremath \mathit{x}}
\newcommand{\stsp}{\ensuremath \mathcal{X}}

\newcommand{\CC}{\mathbb{C}}
\newcommand{\NN}{\mathbb{N}}
\newcommand{\al}{\alpha}

\newcommand{\la}{\lambda}

\newcommand{\mbr}[4]{\mbox{release}_{#1, #2}(#3, #4)}
\newcommand{\rprob}[4]{r_{#1, #2}(#3, #4)}
\newcommand{\mbs}[4]{\mbox{shift}_{#1, #2}(#3, #4)}
\newcommand{\sprob}[4]{s_{#1, #2}(#3, #4)}
\newcommand{\tavail}{\mathcal{T}^{(a)}}
\newcommand{\stOI}{\ensuremath {\mathit{x}}^{(OI)}}

\usepackage[normalem]{ulem}
\usepackage{graphicx,cite}
\usepackage{amssymb,amsmath,amsthm}
\usepackage{amssymb,bm}
\usepackage{array}
\usepackage{amsfonts}
\usepackage{enumitem}
\usepackage{color}

\usepackage{ifthen}

\usepackage{setspace}
\usepackage{amsmath,amsfonts}
\usepackage{amssymb}
\usepackage{graphicx}

\usepackage{amsmath,amscd,amssymb,amsgen,amsfonts,amsbsy,amsthm}
\usepackage{mathrsfs}
\usepackage[colorlinks]{hyperref}
\usepackage{pdfsync}
\usepackage{bm}
\usepackage{bbm}
\usepackage{color}
\usepackage{textcomp}
\usepackage{float}
\usepackage{latexsym,graphicx}
\usepackage{url}
\usepackage{todonotes}
\usepackage{authblk}

\theoremstyle{plain}

\newtheorem{condition}{Condition}
\newtheorem{theorem}{Theorem}
\newtheorem{lemma}[theorem]{Lemma}
\newtheorem{corollary}[theorem]{Corollary}

\theoremstyle{remark}

\newtheorem{remark}{Remark}

\begin{document}
\title{A token-based central queue with order-independent service rates}
\author[1,3,5,6]{U. Ayesta}
\author[2,5]{T. Bodas}
\author[,4]{J.L. Dorsman\thanks{Corresponding author: j.l.dorsman@uva.nl}}
\author[1,5]{I.M. Verloop}
\affil[1]{CNRS, IRIT, 2 rue Charles Camichel, 31071 Toulouse, France}
\affil[2]{CNRS, LAAS, 7 Avenue du colonel Roche, 31400 Toulouse, France}
\affil[3]{IKERBASQUE - Basque Foundation for Science, 48011 Bilbao, Spain}
\affil[4]{Korteweg-de Vries Institute for Mathematics, University of Amsterdam, P.O. Box 94248, 1090 GE Amsterdam, The Netherlands}
\affil[5]{Universit\'e de Toulouse, INP, 31071 Toulouse, France}
\affil[6]{UPV/EHU, University of the Basque Country, 20018 Donostia, Spain}

\maketitle

\begin{abstract}
We study a token-based central queue with multiple customer types. Customers of each type arrive according to a Poisson process and have an associated set of compatible tokens. Customers may only receive service when they have claimed a compatible token. If upon arrival, more than one compatible token is available, an \emph{assignment rule} determines which token will be claimed. The  \emph{service rate} obtained by a customer is state-dependent, i.e., it depends on the set of claimed tokens and on the number of customers in the system. Our first main result shows that, provided the \emph{assignment rule} and the \emph{service rates} satisfy certain conditions, the steady-state distribution has a product form. We show that our model subsumes known families of models that have product-form steady-state distributions including the order-independent queue of \cite{Krzesinski11} and the model of \cite{Visschers12}. Our second main contribution involves the derivation of expressions for relevant performance measures such as the sojourn time and the number of customers present in the system.  We apply our framework to relevant models, including an M/M/K queue with heterogeneous service rates, the MSCCC queue, multi-server models with redundancy and matching models. For some of these models, we present expressions for performance measures that have not been derived before. \\

\end{abstract}



\section{Introduction}
The discovery of queueing systems with a steady-state product-form distribution is probably one of the most fundamental contributions in queueing theory. In a pioneering work,  \cite{Jackson57} showed that in a queueing network formed by  M/M/1 nodes, the joint steady-state distribution is given by the product of the marginal distributions of the individual nodes. 
Roughly speaking, this implies that the stationary distribution of the network can be obtained by multiplying the stationary distributions of the individual nodes assuming that each node is in isolation. Due to this property, the analysis of a queueing network reduces to that of single-node queues, simplifying the analysis  tremendously.  Product-form distributions  provide insight into the impact of parameters on the performance and allow efficient calculation of performance measures. As a consequence, since Jackson's discovery, considerable effort has been put in understanding the conditions such that a stochastic model has a product-form steady-state distribution. An important step forward was made by \cite{BCMP75} and \cite{kelly76}, who introduced BCMP networks and Kelly networks, respectively, which have product-form steady-state distributions. These networks demonstrate that models with multiple types of customers and general service time distributions could also have a product-form distribution. Since then, further studies have shown that networks with negative arrivals, instantaneous signals and blocking might have a product-form distribution, see \cite{Xiuli11} for an overview.

Recent years have witnessed a surge of interest in parallel server models with different types of customers. The main application is in the study of data centers, which consists of a pool of resources that are interconnected by a communication network. Indeed, data centers  provide the main infrastructure to support many internet applications, enterprise operations and scientific computations. In two relevant studies, \cite{Visschers12} and \cite{Krzesinski11}, sufficient conditions have been obtained for a multi-server system to have a product form.  We note that these product-form distributions are not expressed as the product of per-type or per-server terms. In fact, they are expressed as a product of terms that correspond to a unique customer in the system. In that respect, they do not allow an interpretation in terms of a product of marginal distributions, as is the case with classical product-form distributions for Jackson, BCMP and Kelly networks. A notable difference between the two papers is in the state descriptor considered therein. In the multi-type customer and server model of \cite{Visschers12}, the authors consider an aggregated descriptor that keeps track of the servers being active but not of the type of customers being served or waiting. On the other hand, in the order-independent queue of \cite{Krzesinski11}, the state descriptor keeps track of the type of customers in the system, but not of the servers being active. These two modelling approaches have led to two separate streams of papers, where each of the approaches covers applications that are not covered by the other. Some of the applications studied are systems with blocking, redundancy and computer clusters, see Section~\ref{sec:applications} for more details. A natural question that arises is whether the original models of \cite{Visschers12} and \cite{Krzesinski11} can be generalised while preserving the product-form distribution in steady state.

We answer this question in the affirmative in this paper. We analyse a token-based central queue with multiple types of customers and multiple tokens. As will be proved in the paper, this model is a generalisation of both the model of  \cite{Visschers12} and the order-independent queue of \cite{Krzesinski11}. 
Customer of each type arrive according to a Poisson process and have an associated set of compatible tokens. To receive service, a customer must claim a compatible token. Therefore, an arriving customer will immediately claim a compatible token if there is one available, otherwise it will wait until it can claim one. As will become clear later on in the paper, the meaning of a token is application-dependent. It might represent a physical server or it might represent the total service rate allocated to a given  customer type. If upon a customer arrival more than one compatible token is available, an \emph{assignment rule} determines which token will be claimed by the customer.  A customer without a token receives no service and a customer holding a token receives service at a rate given by a state-dependent \emph{service rate function} that satisfies certain conditions. As we will show later, these conditions are reminiscent of those in the order-independent queue. 

Our first main result shows that, provided the \emph{assignment rule} and the \emph{service rate function} satisfy the required conditions, the steady-state distribution has a product form. 
As in the case of \cite{Visschers12} and \cite{Krzesinski11}, this product-form distribution cannot be  expressed as the product of per-type or per-token terms.
We further show that the order-independent  queue and the multi-type customer and server model of \cite{Visschers12}  are particular instances of our model and that our model includes examples that were not covered by either. In other words, our model and main results provide a unifying framework for parallel-server models with a product-form distribution. Our second main contribution is that we use the steady-state distribution of the general model to characterise transforms of relevant performance measures, including the sojourn time and the number of customers in the system.  We illustrate the applicability of the framework by computing the steady-state distribution and analysing the performance of many relevant models, including an $M/M/K$ queue with heterogeneous service rates, the MSCCC queue, multi-server models with redundancy and  matching models. For some of these models, we present expressions for performance measures that have not been derived before. It is important to note that, even though our model is based on a central-queue architecture, some of the applications, in particular the redundancy models, correspond to topologies without a central queue, where instead every server has its own queue. We explain this in more detail in Section~\ref{sec:applications}.

 
 





The rest of the paper is organised as follows. In the next section, we discuss studies related to this paper. Section \ref{sec:model} then describes the token-based central queue that we study in more detail and introduces the required notation. Section \ref{sec:statDist} shows that the token-based central queue has a product-form stationary distribution, which allows for the calculation of other performance measures in Section \ref{sec:perfEval}. We show in Section \ref{sec:generalisation} that the models of \cite{Krzesinski11} and \cite{Visschers12} are captured by our model, after which we discuss applications of our model in Section \ref{sec:applications}.

\section{Related work}\label{sec:related_work}
As mentioned in the introduction, there has been a surge of interest in multi-server queueing models in recent years. The main two references related to our work are \cite{Visschers12} and \cite{Krzesinski11}, which identify classes of models that have a product-form stationary measure. 

Subsequently, several studies have used the results of these two models to analyse a variety of other models.  An important application area that has received a lot of attention is formed by redundancy models. While there are several variants  of a redundancy-based system, the  general notion of redundancy is to  create multiple copies of the same customer that will be sent to a subset
of servers. Depending on when replicas are deleted, there are two classes of redundancy systems: cancel-on-start (COS) and cancel-on-completion (COC).
In redundancy systems with COC, once a copy has completed service, the other copies
are deleted and the customer is said to have received service. On the other hand, in redundancy
systems with COS, copies are removed as soon as one copy starts being served.  
In \cite{Bonald17a}, the authors observe that the COC model is a special case of the order-independent queue \cite{Krzesinski11}, which
enables the authors to derive the steady-state distribution directly. We also refer to  \cite{Gardner16} for a thorough analysis of the COC system. On the other hand, \cite{Ayesta18} shows that while the COS based redundancy system is not an order-independent queue,  it fits within the multi-type customer and server model of \cite{Visschers12}. They also show that, while the COC model does not the framework of \cite{Visschers12}, it does fit an extension of it, where the state descriptor used in \cite{Visschers12} is endowed with a more general service rate function. We will use the resulting state descriptor also in this paper (see Section~\ref{sec:model} for more details).

An important application area, which fits the framework of \cite{Visschers12}, is that of matching models, which have been studied in several recent papers, see for instance \cite{Adanmatching18}. We also refer to \cite{AdanRighterWeiss1} and \cite{AdanRighterWeiss2}, where the authors explore the relation between redundancy and matching models. In Section~\ref{sec:applications}, we apply our token-based approach to derive the steady-state distribution of a large family  of  matching models.

Another important related work is \cite{AdanWeiss}. The model considered therein is similar to the one of \cite{Visschers12} with the exception that the assignment policy \emph{`assign longest idle server'} (ALIS) is used. Under the ALIS-policy, a new arrival that could be served by more than one inactive server, is assigned to the longest-idle server. To implement this policy, the state descriptor is enriched with information on the idleness of every inactive server. 
The authors prove that the steady-state distribution of this model has a product form.  In our paper, we do not consider the ALIS variant, however, from the analysis of \cite{AdanWeiss}, we expect that all our results would carry over to this case. We discuss this in more detail in Section~\ref{sec:statDist}.

 
 
 




\section{Model description}\label{sec:model}

In this section, we describe the token-based central queue model in more detail.

\emph{Customers and tokens.\;} The model that we study represents a central-queue system where the customers may be of mutually different types (or classes). 
The set of all customer types is denoted by $\mathcal{C}$ and customers of type $c \in \mathcal{C}$ arrive according to a Poisson process with rate $\lambda_c$. As a result, the total arrival rate of customers to the system is $\lambda := \sum_{c \in \mathcal{C}}\lambda_c$. A distinguishing feature of this model is the fact that in order for customers to receive service, they must hold a \emph{token}. To this end, a set of $K$ tokens denoted by $\mathcal{T} = \{t_1,\ldots,t_K\}$ is also associated with the model. In particular, a customer type $c \in \mathcal{C}$ is characterised by a token set $\mathcal{T}_c$ which consists of the compatible tokens that can be held by customers of type $c$. Similarly, associated with a token $t \in \mathcal{T}$ is a set of customer types that can choose the token, denoted henceforth by $\mathcal{C}_t$. Clearly, $\mathcal{T}_c \subseteq \mathcal{T}$ and $\mathcal{C}_t \subseteq \mathcal{C}.$ \newline

\emph{Assignment of customers to tokens.\;} 
At any point in time, the set of available tokens is denoted by $\tavail$, $\tavail \subseteq \mathcal{T}$, while the set of unavailable tokens is given by $\mathcal{T} \backslash \tavail$. To receive service, customers are required to claim a compatible token. Hence, when a customer of type $c \in \mathcal{C}$ arrives, it will claim a single token from the set $\mathcal{T}_c\cap \tavail$ (if it is non-empty), and then join the central queue. In case no compatible token is available upon arrival (i.e.\ $|\mathcal{T}_c\cap \tavail|=0$), the customer will join the queue and wait until a token in the set $\mathcal{T}_c$ becomes available. If multiple compatible tokens are available, i.e., $|\mathcal{T}_c \cap \tavail| > 1,$ then an \textit{assignment rule} decides which of the tokens will be claimed by the arriving customer. More particularly, this assignment rule constitutes a randomised policy which, given $\tavail$ and the class of the arriving customer, dictates the probability with which the customer should claim a particular token. We assume this assignment rule to satisfy a so-called \textit{assignment condition}, which we elaborate on later in this section.  Once a token $t$ is selected by a customer, it is no longer available for selection (i.e. $\tavail := \tavail \backslash \{t\}$) until the customer has completed service. Upon release, the token will immediately be reclaimed by the longest waiting tokenless customer of a type from the set $\mathcal{C}_{t}$. If there are no such customers, the token is added back to the set $\tavail$ such that $\tavail := \tavail \cup \{t\}$. We shall refer to customers with tokens as \textit{active customers} and identify such customers with the token associated with them. Customers in the central queue without tokens will be referred to as \textit{inactive customers}. \newline

\emph{Departure rates of customers.\;} 
We assume service requirements of customers to be exponentially distributed. In light of the model's token mechanism, this means that the departure rate of active customers from the system is non-negative, while that of inactive customers is zero. Throughout the paper, we assume that the departure rates associated with active customers satisfy a certain condition, which is specified below. Since this condition is reminiscent of the order-independent queue as introduced in \cite{Krzesinski11}, we call this the \emph{order-independent condition.} 
\newline

\emph{Markovian state descriptor.\;}
Due to the memoryless properties of the arrival and departure processes, the token-based central queue can be interpreted as a Markov process. We now introduce the state descriptor that we use to analyse this model. We will show in Section \ref{sec:transitionRatesBalanceEquations} that this state descriptor leads to a Markov system by stating its balance equations. The state descriptor that we use for the token-based central queue is of the form $({T}_{1}, n_1,\ldots, T_i, n_i)$.  This descriptor retains the order of arriving customers in the central queue from left to right. When the model is in state $({T}_{1}, n_1,\ldots,{T}_{i-1}, n_{i-1}, T_i, n_i)$, it has $i$ active customers which have claimed tokens $T_1, \ldots, T_i$. Furthermore, there are $n_j$ {inactive customers} in the central queue that have arrived between the two customers that have claimed tokens $T_j$ and $T_{j+1}$, respectively, for $1\leq j \leq i-1$. Inactive customers at the end of the queue are denoted by $n_i$. Since tokens are always claimed by the longest waiting eligible customer, we have that e.g. $n_1$ represents inactive customers which have token $T_1$ as their only compatible token. The set of such customer types is denoted by $\mathcal{U}(\{ T_1\}) := \{c \in \mathcal{C}: \mathcal{T}_c=\{T_1\}\}$. In general, for $1 \leq j \leq i$, we denote the set of customer types that can claim tokens only from the set $\{T_1, \ldots, T_j\}$ by $\mathcal{U}(\{ T_1, T_2, \ldots, T_j\}) := \{c \in \mathcal{C}: \mathcal{T}_c \subseteq \{ T_1,\ldots T_j \} \}$. Thus, the customer types of the $n_j$ customers between those with tokens $T_j$ and $T_{j+1}$ must belong to the set $\mathcal{U}(\{ T_1, T_2, \ldots, T_j\})$. As the state descriptor retains the order of arrival, the oldest customer in a state is represented by token $T_1$. The newest customer is one of the $n_i$ customers, or in case $n_i=0$, it is the active customer with token $T_i$. Furthermore, when $1 \leq j < k \leq i$, all the $n_j$ customers between $T_j$ and $T_{j+1}$ have arrived before the $n_k$ customers between $T_{k}$ and $T_{k+1}$. We henceforth denote the state space of the resulting Markov process by $\stsp$, where any generic state $\st \in \stsp$ is of the type $\st = ({T}_{1}, n_1, \ldots, T_i, n_i)$. The only exception is the empty state with no customers present, which we denote by $(0)$. \newline

\emph{Assignment rule and assignment condition.\;} 
Recall that in case multiple compatible tokens are available upon the arrival of a customer, the \textit{assignment rule} of the system determines the probability with which any of these tokens is assigned to the customer. Furthermore, in state $\st = ({T}_{1}, n_1, \ldots, T_i, n_i)$, the arrival rate of \textit{customers} that will initially be inactive is given by $\lambda_{\mathcal{U}(\{{T}_1,\ldots,{T}_i\})} := \sum_{c \in \mathcal{U}(\{{T}_1,\ldots,{T}_i\})}\lambda_c$, while the arrival rate of \textit{customers} that become active immediately is given by $\lambda - \lambda_{\mathcal{U}(\{{T}_1,\ldots,{T}_i\})}$. Given the nature of the assignment rule, we denote by $\lambda_{t}(\{{T}_1,\ldots,{T}_j\})$ the rate at which arriving customers claim token $t$, provided that $\{T_1, \ldots, T_j\}$ is the set of all unavailable tokens. 
While $\lambda_{t}(\{{T}_1,\ldots,{T}_j\})$ depends on the assignment rule, it holds for any assignment rule that 
\begin{equation}
\label{eq:lambda}
\lambda - \lambda_{\mathcal{U}(\{{T}_1,\ldots,{T}_i\})} = 
\sum_{{t}\in \mathcal{T}\backslash\{{T}_1,\ldots,{T}_i\} } \lambda_{{t}}(\{{T}_1,\ldots,{T}_i\}).
\end{equation}
As in \cite{Visschers12}, for the system to have a product-form stationary distribution, we require that any assignment rule satisfies the following assignment condition.
\begin{condition}\label{cdn:assignment}
An assignment rule is said to satisfy the assignment condition if for any possible combination of $i$ unavailable tokens $T_1, \ldots, T_i$, $i=1, \ldots, K$, it holds that
\begin{equation} 
\prod_{j=1}^{i} \lambda_{T_j}(\{T_1, \ldots T_{j-1} \}) = \prod_{j=1}^{i} \lambda_{\bar{T}_j}(\{\bar{T}_1, \ldots \bar{T}_{j-1} \})
 \end{equation}
for every permutation $\bar{T}_1, \ldots \bar{T}_{i}$ of $T_1, \ldots T_{i}.$
\end{condition}
It is shown in \cite{AdanHurkensWeiss} that there always exists at least one assignment rule for which the assignment condition is satisfied. As we will also see in Section \ref{sec:applications}, the assignment condition generally allows for a rather large set of assignment rules. \newline

\emph{Order-independent condition.\;} To state the order-independent condition, we require additional notation. For any state $\st = (T_1, n_1, \ldots, T_i, n_i)$, let $\mu_{T_j}(\st)$ denote the departure rate of the active customer holding token $T_j$. Furthermore, let $\mu(\st) := \sum_{j=1}^{i} \mu_{T_j}(\st)$ denote the total departure rate in state $\st$. Additionally, we denote by $\phi(\st) = i+\sum_{j=1}^i n_j$ the total number of customers in state $\st$. The order-independent condition now reads as follows.

\begin{condition}\label{cdn:OI}
The departure rates of the model are said to satisfy the order-independent condition if in a given state $\st = (T_1, n_1, \ldots, T_i, n_i)$, each of the rates $\mu_{T_j}(\st)$, $j=1, \ldots, i$, can be written as
 \begin{equation}
  \mu_{T_j}(\st) = \eta(\phi(\st)) s_j(T_1, \ldots, T_i),
 \end{equation}
where
\begin{enumerate}
\item $s_j(\cdot)$ is a non-negative real-valued function for which $s_j(T_1, \ldots, T_i) = s_j(T_1, \ldots, T_j)$, $1 \leq j \leq i$,
\item $k(T_1, \ldots, T_i) := \sum_{j=1}^{i} s_j(T_1, \ldots, T_j)$ is independent of any permutation of $(T_1, \ldots, T_i)$ and
\item $\eta(\cdot)$ is a non-negative real-valued function for which $\eta(j) > 0$ for $j=1,2, \ldots$. 
\end{enumerate}
\end{condition}
These restrictions on the functions $s_j(\cdot)$, $k(\cdot)$ and $\eta(\cdot)$ have the following implications. First, by the restriction $s_j(\cdot)$, the departure rate of an active customer may depend on the types of the active customers ahead of it, but not on those behind. Note that $s_j(\cdot)$ may equal zero, so that it is possible for active customers to still receive no service. Second, $k(\cdot)$ is defined such that the total departure rate of customers from the system is the same for any permutation of the active customers. Finally, the function $\eta(\cdot)$ allows the departure rate of customers to depend on the total number of customers present in the system, but at the same time the departure rate is indifferent to the types of the inactive customers. Next, based on the definition of $\mu(\st)$, we conclude that 
\begin{equation}\label{eq:mu}
\mu(\st) = \eta(\phi(\st))k(T_1, \ldots, T_i).
\end{equation}
As mentioned earlier, this order-independent condition is reminiscent of the order-independent queue as introduced in \cite{Krzesinski11}. The difference, however, stems from the fact that we consider a different state descriptor, which captures a broader set of systems (cf.\ Section \ref{sec:OI}). It is also important to note that this condition allows our model to be more general than that of \cite{Visschers12}, as will become clear in Section \ref{sec:Visschers}. \newline

\emph{Further notation.\;} We conclude this section with notation needed to describe several important performance measures. At an arbitrary point in time, let $N$ denote the number of inactive customers in the system. More particularly, $N_j$ denotes the number of inactive customers in the central queue between the two customers that have claimed tokens $T_j$ and $T_{j+1}$. Thus, when the system is in state $\st = (T_1, n_1, \ldots, T_i, n_i)$, it holds that $N_j = n_j$ and $N = \sum_{j=1}^i n_j$. Moreover, the number of type-$c$ customers among these $N_j$ customers is denoted by $N_j^{(c)}$. As a consequence, the total number of inactive type-$c$ customers, denoted by $N^{(c)}$, satisfies $N^{(c)} = \sum_{j=1}^K N_j^{(c)}$. Using the same style of notation, $M$ denotes the total number of customers present in the system. Furthermore, for $1\le j \le i$, $M_j = N_j+1$ represents the number of customers in the `$j$-th' part of the system, where the added single customer is the one that holds token $T_j$. Of these $M_j$ customers, $M_j^{(c)}$ are of type $c$, so that $M^{(c)}$, the number of type-$c$ customers present in the system, satisfies $M^{(c)} = \sum_{c \in \mathcal{C}} M_j^{(c)}$. Next, we define the time-till-token of a customer to be the duration of the period between its arrival and the moment the customer claims of a token. Then,  the \emph{time-till-token} and the sojourn time of a type-$c$ customer is denoted by $W_c$ and $S_c$, respectively. Likewise, the quantities $W$ and $S$ refer to the time-till-token and the sojourn time of an arbitrary customer. 
Finally, the indicator function $\ind{A}$ on the event $A$ returns one if event $A$ is true, and zero otherwise.

\section{Product-form stationary distribution}\label{sec:statDist}
In this section, we derive the stationary distribution of the token-based central queue and find that it has a product form.
In doing so, we use techniques from \cite{Visschers12}. We describe the transition rates and the global balance equations pertaining to this process in Section \ref{sec:transitionRatesBalanceEquations}. Then, we proceed to derive the product-form stationary distribution in Section \ref{sec:productFormStatDist}. Section \ref{sec:L} subsequently points out how this stationary distribution can be computed efficiently for models with indistinguishable tokens by aggregation of states. We conclude with a note on the stability conditions of the token-based central queue in Section \ref{sec:stability}.

\subsection{Transition rates and balance equations}\label{sec:transitionRatesBalanceEquations}
The transitions associated with the model are organised into the following three categories. The first category contains transitions that are caused by the arrival of a customer. The second and third categories pertain to transitions due to a departure of a customer (i.e.\ a completion of service). In particular, the second category contains departure transitions where a token becomes available. The third category contains the remaining departure transitions, where the released token is immediately reclaimed by an inactive customer. We now proceed to describe the transition rates within each of these categories. 

\subsubsection{Arrival transitions}\label{sec:arrivalRates}
Recall from Section \ref{sec:model} that an arriving customer either joins the central queue as an inactive customer (when it finds no compatible tokens in the set $\tavail$) or joins it as an \textit{active customer}. In a given state $\st = (T_1, n_1, \ldots, T_i, n_i)$, the arrival rate of inactive customers is $\lambda_{\mathcal{U}(\{{T}_1,\ldots,{T}_i\})}$. Hence, this is also the transition rate from state $\st$ to state $(T_1, n_1, \ldots, T_i, n_i+1)$. Customers that immediately claim a token $t$ upon arrival, $t\notin\{T_1, \ldots, T_i\}$, arrive at rate $\lambda_t(\{T_1, \ldots, T_i\})$. Therefore, the transition rate from state $\st$ to state $(T_1, n_1, \ldots, T_i, n_i, t)$ is given by $\lambda_t(\{T_1, \ldots, T_i\})$. Recall that these rates satisfy Condition \ref{cdn:assignment}. By virtue of \eqref{eq:lambda}, we conclude that the total arrival rate into the system in any given state equals $\lambda$, as expected.

\subsubsection{Departure transitions where tokens become available}\label{sec:departureRatesRelease}
To describe the departure rates, additional notation is required. Transitions to a state $\st = (T_1, n_1, \ldots, T_i, n_i)$ due to a departure of a customer where a token $T$ is released to the set $\tavail$ are possible from states of the form
\begin{equation*}
\mbr{k}{n}{\st}{T} = (T_1, n_1, \ldots, T_k, n_k-n, T, n, T_{k+1}, n_{k+1}, \ldots, T_i, n_i),
\end{equation*}
where $ k \in \{0, \ldots, i\}$, $n \in \{0, \ldots, n_k\}$ and $T \in \mathcal{T}\backslash\{T_1, \ldots, T_i\}$.
For future reference, it is worth noting that $\phi(\mbr{k}{n}{\st}{T}) = \phi(\st)+1$, as the population size of the two states only differ by the single active customer that releases token $T$. Furthermore, we have that 
\begin{align}
\mu_T(\mbr{k}{n}{\st}{T}) &= \eta(\phi(\st)+1) s_T(T_1, \ldots, T_k, T, T_{k+1}, \ldots, T_i) \notag\\
&=\eta(\phi(\st)+1)\left(k(T_1,\ldots,T_k, T) - k(T_1,\ldots,T_k)\right). 		\label{eq:releaseRate}
\end{align}
Note, however, that $\mu_T(\mbr{k}{n}{\st}{T})$ is merely the rate in state $\mbr{k}{n}{\st}{T}$ at which the customer holding $T$ leaves the system. To obtain the transition rate from state $\mbr{k}{n}{\st}{T}$ to $\st$, this quantity must be multiplied with the probability that after this departure, the token $T$ is indeed released from activity. This probability is given by $\rprob{k}{n}{\st}{T} = \beta_k(T)^n\beta_{k+1}(T)^{n_{k+1}}\cdots\beta_i(T)^{n_i}$, where
\begin{equation}\label{eq:beta}
\beta_k(T) = \frac{\lambda_{\mathcal{U}(\{T_1, \ldots, T_k\})}}{\lambda_{\mathcal{U}(\{T_1, \ldots, T_k, T\})}}
\end{equation}
is the probability that a customer waiting in the $k$-th portion of the central queue can not be served by token $T$. As a special case, we define $\beta_0(T) = 0$ for any token $T\in \mathcal{T}$. It now follows that the transition rate from $\mbr{k}{n}{\st}{T}$ to $\st$ is given by $\mu_T(\mbr{k}{n}{\st}{T})\rprob{k}{n}{\st}{T}$.

\subsubsection{Departure transitions where tokens are reassigned}\label{sec:departureRatesReassign}
We proceed to consider the departure transitions, where a token is immediately reclaimed by another customer waiting further down the central queue. Transitions of this type to state $\st = (T_1, n_1, \ldots, T_i, n_i)$ are possible from states of the form
\begin{equation*}
\mbs{k}{n}{\st}{T_j} = (T_1, n_1, \ldots, T_k, n_k-n, T_j, n, T_{k+1}, n_{k+1}, \ldots, T_{j-1}, n_{j-1}+1+n_j, T_{j+1}, n_{j+1}, \ldots, T_i, n_i),
\end{equation*}
where $1 \le k \le i$, $k+1 < j \le i$ and $n \in \{0, \ldots, n_k\}$. These transitions describe the event that the customer which holds $T_j$ departs the system, and token $T_j$ is subsequently reclaimed by an inactive customer between the customers holding $T_{j-1}$ and $T_{j+1}$. Again, $\phi(\mbs{k}{n}{\st}{T}) = \phi(\st)+1$ and 
\begin{align}
\mu_T(\mbs{k}{n}{\st}{T_j}) &= \eta(\phi(\st)+1) s_{T_j}(T_1, \ldots, T_k, T_j, T_{k+1}, \ldots, T_i)\notag\\
&=\eta(\phi(\st)+1)\left(k(T_1, \ldots, T_k, T_j) - k(T_1,\ldots,T_k)\right). 		\label{eq:shiftRate}
\end{align}
Similar to the previous case, the transition rate from a state $\mbs{k}{n}{\st}{T_j}$ to state $\st$ can be argued to be equal to $\mu_T(\mbs{k}{n}{\st}{T_j})\sprob{k}{n}{\st}{T_j}$, where the latter factor is given by
\begin{equation*}
\sprob{k}{n}{\st}{T_j} = \beta_k(T_j)^n\beta_{k+1}(T_j)^{n_{k+1}}\cdots\beta_{j-1}(T_j)^{n_{j-1}}(1-\beta_{j-1}(T_j)),
\end{equation*}
with $\beta_k(T_j)$ as defined in \eqref{eq:beta}.

\subsubsection{Global balance equations}
Now that all transitions rates have been described, the global balance equations can be obtained. Using results from Sections \ref{sec:arrivalRates}-\ref{sec:departureRatesReassign}, denoting the stationary distribution by $\{\pi(\st): \st\in\stsp\}$ and recalling that the total departure rate from a state $\st$ is simply $\mu(\st)$, the global balance equations are, for $\st=(T_1, n_1, \ldots, T_i, n_i)\in\stsp\backslash\{(0)\}$, given by
\begin{align}
 (\la+\mu(\st))\pi(\st) =\;& \ind{n_i>0}\la_{\mathcal{U}(\{T_1, \ldots, T_i\})}\pi(T_1, n_1, \ldots, T_i, n_i-1) \notag \\
&+\ind{n_i=0}\la_{T_i}(\{T_1, \ldots, T_{i-1}\})\pi(T_1, n_1, \ldots, T_{i-1}, n_{i-1}) \notag\\
&+ \sum_{T \in \mathcal{T}\backslash\{T_1, \ldots, T_i\}}\sum_{k=0}^i\sum_{n=0}^{n_k} \mu_T(\mbr{k}{n}{\st}{T})\rprob{k}{n}{\st}{T}\pi(\mbr{k}{n}{\st}{T}) \notag \\
&+ \sum_{j=1}^i\sum_{k=0}^{j-1}\sum_{n=0}^{n_k} \mu_{T_j}(\mbs{k}{n}{\st}{T_j})\sprob{k}{n}{\st}{T_j}\pi(\mbs{k}{n}{\st}{T_j}). \label{eq:globalBalanceEquations}
\end{align}
For $\st = (0)$, however, several of these terms can be omitted, so that
\begin{equation}\label{eq:globalBalanceEquationsEmptyState}
\lambda \pi((0)) = \sum_{T \in \mathcal{T}} \mu_T((0, T))\pi((0,T)).
\end{equation}

%

\subsection{Product-form stationary distribution}\label{sec:productFormStatDist}
We now present one of the main contributions of this paper. When both the assignment condition and the order-independent condition (cf.\ Conditions \ref{cdn:assignment} and \ref{cdn:OI}) are satisfied, the token-based central queue together with its state descriptor allows for a product-form stationary distribution. This distribution is given in the following theorem.
\begin{theorem}
\label{thm:statDist}
If the token-based central queue is stable and Conditions \ref{cdn:assignment} and \ref{cdn:OI} are satisfied, then, for each $\st = (T_1, n_1, \ldots, T_i, n_i) \in \stsp$, the stationary distribution is given by 
\begin{equation} \label{eq:statDist}
 \pi(\st) = \pi((0))\frac{\Pi_{\lambda}(\{{T}_1, \ldots, {T}_i\})}{\Pi_{k}({T}_1,\ldots,{T}_i)}\prod_{j=1}^i{\alpha_j}^{n_j} \prod_{j=1}^{\phi(\st)}\frac{1}{\eta(j)},
\end{equation}
where
\begin{equation*}
\Pi_{\lambda}(\{{T}_1, \ldots, {T}_i\}) = \prod_{j=1}^{i} \lambda_{{T}_j}(\{{T}_1, \ldots, {T}_{j-1} \}), \Pi_{k}({T}_1,\ldots,{T}_i) = \prod_{j=1}^{i} k({T}_1, \ldots, {T}_{j})\mbox{ and }\alpha_j = \frac{\lambda_{\mathcal{U}(\{{T}_1,\ldots,{T}_j\})}}{k({T}_1, \ldots {T}_{j})}.
\end{equation*}
The normalising constant $\pi((0))$ is given by 
\begin{equation}\label{eq:piZero}
\pi((0)) = \left(1+\sum_{i=1}^K \sum_{(T_1, \ldots, T_i)\in \mathcal{T}^{i}}\frac{\Pi_{\lambda}(\{{T}_1, \ldots, {T}_i\})}{\Pi_{k}({T}_1,\ldots,{T}_i)}\sum_{(n_1, \ldots, n_i)\in \mathbb{N}^i}\prod_{j=1}^i{\alpha_j}^{n_j} \prod_{j=1}^{i+\sum_{k=1}^i n_k}\frac{1}{\eta(j)}\right)^{-1},
\end{equation}
where $\mathcal{T}^{i}$ denotes the set of all possible combinations of $i$  tokens from the set $\mathcal{T}$.
\end{theorem}

\begin{proof}
This proof verifies that \eqref{eq:statDist} satisfies the global balance equations \eqref{eq:globalBalanceEquations} and \eqref{eq:globalBalanceEquationsEmptyState}, which guarantees that \eqref{eq:statDist} represents the unique stationary distribution. It is straightforward to show that \eqref{eq:statDist} satisfies \eqref{eq:globalBalanceEquationsEmptyState}. To see that \eqref{eq:statDist} satisfies \eqref{eq:globalBalanceEquations}, we show in Appendix \ref{app:thm:statDist} that \eqref{eq:statDist} satisfies the following three equations for every $\st \in \stsp$ and $T\in\mathcal{T}\backslash\{T_1, \ldots, T_i\}$:
\begin{align}\label{eq:balanceEquationArrival}
 \mu(\st)\pi(\st) =\;& \ind{n_i>0}\la_{\mathcal{U}(\{T_1, \ldots, T_i\})}\pi((T_1, n_1, \ldots, T_i, n_i-1)) \notag \\
&+\ind{n_i=0}\la_{T_i}(\{T_1, \ldots, T_{i-1})\pi(T_1, n_1, \ldots, T_{i-1}, n_{i-1}),
\end{align}
\begin{equation}\label{eq:balanceEquationDepartureRelease}
\la_{T}(\{T_1, \ldots, T_{i}\})\pi(\st) =\sum_{k=0}^i\sum_{n=0}^{n_k} \mu_T(\mbr{k}{n}{\st}{T})\rprob{k}{n}{\st}{T}\pi(\mbr{k}{n}{\st}{T})
\end{equation}
and
\begin{equation}\label{eq:balanceEquationDepartureShift}
\la_{\mathcal{U}(\{T_1, \ldots, T_i\})}\pi(\st) = \sum_{j=1}^i\sum_{k=0}^{j-1}\sum_{n=0}^{n_k} \mu_{T_j}(\mbs{k}{n}{\st}{T_j})\sprob{k}{n}{\st}{T_j}\pi(\mbs{k}{n}{\st}{T_j}).
\end{equation}
Summing \eqref{eq:balanceEquationDepartureRelease} over all available tokens $T \in \mathcal{T}\backslash\{T_1, \ldots, T_i\}$ and adding \eqref{eq:balanceEquationArrival} and \eqref{eq:balanceEquationDepartureShift}, we conclude using \eqref{eq:lambda} that \eqref{eq:statDist} satisfies \eqref{eq:globalBalanceEquations}. The theorem now follows. 
\end{proof}

\begin{remark}\label{rem:closedFormPiZero}
Note that the expression for the stationary distribution in \eqref{eq:statDist} is not in closed form. This is due to the fact that the normalising constant $\pi((0))$ contains infinite sums. For some specific cases of the function $\eta(\cdot)$, though, given that the token set is finite, $\pi((0))$ allows for a closed-form expression. For example, when $\eta(\cdot)=1$, \eqref{eq:piZero} reduces to
\begin{equation}\label{eq:piZeroSimple}
\pi((0)) =\left(1+\sum_{i=1}^K \sum_{(T_1, \ldots, T_i)\in \mathcal{T}^i}\frac{\Pi_{\lambda}(\{{T}_1, \ldots, {T}_i\})}{\Pi_{k}({T}_1,\ldots,{T}_i)}\prod_{j=1}^i \frac{1}{1-\alpha_j}\right)^{-1},
\end{equation}
which is in closed form. We will see in Section \ref{sec:applications} that $\eta(\cdot)$ is a constant function in many applications.
\end{remark}

\begin{remark}
While we assume our model to satisfy Condition \ref{cdn:assignment}, a different assignment mechanism has been studied in \cite{AdanWeiss} called ALIS: `Assign Longest Idle Server'. Stated in the context of the token-based central queue, the key feature of an ALIS queue is that an arriving customer who finds multiple eligible tokens upon arrival, will activate the token that has been available the longest. Since this mechanism cannot be captured by an assignment rule as described in Section \ref{sec:model}, we must extend the state descriptor to keep track of which token has been available the longest, in order to regard this mechanism. The new state descriptor is of the form $(T_1, n_1, \ldots, T_i, n_i, T_{i+1}, \ldots, T_K)$, where $T_{i+1}, \ldots, T_{K}$ are the available tokens in ascending order of the time they have been idle. In other words, if an arriving customer is eligible to claim token $T_K$, it will do so. Otherwise, it will claim $T_{K-1}$ if it is able to do so, and so on. By using the same proof techniques as in \cite{AdanWeiss}, we expect it can be shown that the stationary distribution for the token-based central queue with an ALIS mechanism also has a product form.
\end{remark}

\subsection{Aggregation of states for indistinguishable tokens}\label{sec:L}
When the model contains tokens which are indistinguishable from one another, computation of the stationary distribution in Theorem \ref{thm:statDist}, and especially its normalising constant in \eqref{eq:piZero} can be made more efficient. To define the notion of indistinguishability, we write the token set $\mathcal{T}$ as a union of disjoint token sets $\tilde{\mathcal{T}}_1, \ldots,\tilde{\mathcal{T}}_l$, where it holds for any two tokens $s,t\in \tilde{\mathcal{T}}_i$, $i \in \{1, \ldots, l\}$ that $\mathcal{C}_s = \mathcal{C}_t$, $\lambda_s({T_1, \ldots, T_j}) = \lambda_t({T_1, \ldots, T_j})$, $\lambda_{T_j}(\{T_1, \ldots, T_{j-1}, s\}) = \lambda_{T_j}(\{T_1, \ldots, T_{j-1}, t\})$ and $k(T_1, \ldots, T_j, s) = k(T_1, \ldots, T_j, t)$ for $T_1, \ldots, T_j \in\mathcal{T}\backslash\{s,t\}$. We then call tokens which belong to the same token set $\tilde{\mathcal{T}}_i$ indistinguishable from one another.

If the number of disjoint token sets $l$ is much smaller than $|\mathcal{T}|$, the computational burden of Theorem \ref{thm:statDist} can be relieved by state aggregation. To this end, let us say that a token $t$ has a token label $l_k$ whenever $t \in \tilde{\mathcal{T}}_k$, $k \in \{1, \ldots, k\}$. Then, since any two tokens $s$ and $t$ from the set $\tilde{\mathcal{T}}_k$ are indistinguishable, we may as well address both of them by their token label $l_k$. This leads to the state descriptor of the form $\st^{(L)} = (L_1, n_1, \ldots, L_i, n_i)$, where $L_i$ now represents the label of the token which is held by the $i$-th active customer in the system, but not the identity of the actual token. We denote the state space under this state descriptor with $\stsp^{(L)}$.  
Let $l(t)$ denote the label of token $t\in \mathcal{T}$, i.e. $l(t) = l_j$ if $t \in \tilde{\mathcal{T}}_j$. Then, by aggregation of states, one can derive the following stationary distribution for the aggregated state descriptor from \eqref{eq:statDist}:
\begin{align}
\pi((L_1, n_1, \ldots, L_i, n_i)) &= \sum_{(T_1, \ldots, T_i)\in\mathcal{T}^i: l(T_j) = L_j \forall j \in\{1, \ldots, i\}} \pi((T_1, n_1, \ldots, T_i, n_i)) \notag \\
&= \pi((0))\prod_{j=1}^i \frac{\lambda_{L_j}(\{L_1, \ldots, L_j\})}{k(L_1, \ldots, L_j)}\prod_{j=1}^i \left(\frac{\lambda_{\mathcal{U}(\{L_1, \ldots, L_j\})}}{k(L_1, \ldots, L_j)}\right)^{n_j} \prod_{j=1}^{\phi(\st)}\frac{1}{\eta(j)}, \label{eq:statdistL}
\end{align}
where $\lambda_{L_j}(\{L_1, \ldots, L_{j-1}\}) = \sum_{t\in\mathcal{T}:l(t) = L_j}\lambda_{t}(T_1, \ldots, T_{j-1})$ (with $l(T_1) = L_1, l(T_2) = L_2, \ldots$) represents the arrival rate of customers that immediately claim a token with label $L_j$, when there are $j-1$ active customers that have claimed tokens from labels $L_1, \ldots, L_{j-1}$.
Likewise, when considering tokens $T_1, \ldots, T_i$ such that $L_j = l(T_j)$ for $j \in \{1, \ldots, i\}$, we define $k(L_1, \ldots, L_j)=k(T_1, \ldots, T_j)$ and $\mathcal{U}(\{L_1, \ldots, L_j\})$ represents the same set of customer types as $\mathcal{U}(\{T_1, \ldots, T_j\})$. The normalising constant $\pi((0))$ as given in \eqref{eq:piZero} remains unchanged, but can now alternatively be written as
\begin{equation*}
\pi((0)) = \left(1+\sum_{i=1}^K \sum_{(L_1, \ldots, L_i)\in \mathcal{L}^{i}}\prod_{j=1}^i \frac{\lambda_{L_j}(\{L_1, \ldots, L_j\})}{k(L_1, \ldots, L_j)}\prod_{j=1}^i \left(\frac{\lambda_{\mathcal{U}(\{L_1, \ldots, L_j\})}}{k(L_1, \ldots, L_j)}\right)^{n_j} \prod_{j=1}^{\phi(\st)}\frac{1}{\eta(j)}\right)^{-1},
\end{equation*}
where $\mathcal{L}^{i}$ represents any possible combination of $i$ token labels. In the sequel, when working with the aggregated state descriptor, we will use $\mu_{L_j}(\st^{(L)})$ and $s_j(L_1, \ldots, L_j)$ as notation for the equivalents of $\mu_{T_j}(\st)$ and $s_j(T_1, \ldots, T_j)$.

\subsection{Stability}\label{sec:stability}
From the stationary distribution \eqref{eq:statDist}, conditions for stability can be derived. In particular, in case the function $\eta(\cdot)$ has a limit $\eta := \lim_{j \rightarrow \infty} \eta(j)$, the system will be stable if $\frac{\lambda_{\mathcal{U}(\{T_1, \ldots, T_i\})}}{k(T_1, \ldots, T_i)} < \eta$ for each $i\in \{1, \ldots, K\}$ and $\{T_1, \ldots, T_i\} \subset T$. Under this condition, Equation \eqref{eq:statDist} constitutes a non-null and convergent solution of the equilibrium equations of the irreducible Markov process underlying the model. As such, it is implied by \cite[Theorem 1]{Foster} that the Markov is ergodic, leading to stability.
When $\frac{\lambda_{\mathcal{U}(\{T_1, \ldots, T_i\})}}{k(T_1, \ldots, T_i)} > \eta$ for some $i\in \{1, \ldots, K\}$ and $\{T_1, \ldots, T_i\} \subset T$, we have by \eqref{eq:piZero} that $\pi((0)) = 0$, implying that the expected return time to state (0) is infinite. As such, the Markov process is not ergodic and the token-based central queue is unstable. 
Finally, in case $\max \frac{\lambda_{\mathcal{U}(\{T_1, \ldots, T_i\})}}{k(T_1, \ldots, T_i)} = \eta$, the questions whether or not there is ergodicity depends on the way (and possibly the speed at which) the function $\eta(\cdot)$ converges to its limit $\eta$.

\section{Performance analysis}\label{sec:perfEval}
Now that we have derived the (product-form) stationary distribution in Section \ref{sec:statDist}, we study several performance measures of the token-based central queue. In particular, we study the (per-type) number of inactive customers in Section \ref{sec:numberCustWaiting}. Likewise, we study the (per-type) number of customers present in the system in Section \ref{sec:numberCustPresent}. Then, making use of the distributional form of Little's law (cf.~\cite{KeilsonServi}), we obtain results for the time-till-token $W_c$ of type-$c$ customers (i.e., the time it takes for customers to claim a token) in Section \ref{sec:waitingTime}. As we will see in Section \ref{sec:applications}, $W_c$ coincides with the waiting time of type-$c$ customers in many applications of our model. Finally, we also consider the sojourn time of customers in Section \ref{sec:sojournTime}.

\subsection{Number of inactive customers}\label{sec:numberCustWaiting}
This section considers the number of inactive customers in the system. For applications where the time-till-token represents the waiting time, this number coincides with the number of customers in the system waiting for service. The main theorem of this section concerns the probability generating function (PGF) of $N^{(c)}$, the number of type-$c$ customers that are inactive.
\begin{theorem}\label{thm:jointPGFNumberCustWaitingClass}
Let $\theta_{c,j} := \frac{\lambda_c\ind{c \in \mathcal{U}(T_1, \ldots, T_j)}}{\lambda_{\mathcal{U}({T_1, \ldots, T_j}})}$ for $j \in \{1, \ldots, K\}$ and $c \in \mathcal{C}$. Then, the joint PGF of $\{N^{(c)}: c \in \mathcal{C}\}$ is, for $z_{c} \in \{\bar{c}\in\CC:|\bar{c}|<1\}$, given by
\begin{align}
\E{\prod_{c \in \mathcal{C}}z_c^{N^{(c)}}}&=\sum_{i=0}^K \sum_{(T_1, \ldots, T_i)\in \mathcal{T}^{i}} \pi((0))\frac{\prodlambda}{\prodk}  \times\notag\\
&\qquad\qquad\qquad\times\prod_{j=1}^i\frac 1{\eta(j)}\sum_{\{n_1, \ldots, n_i\} \in \NN_0^i}\prod_{j=1}^{\sum_{k=1}^i n_k} \frac 1{\eta(i+j)} \prod_{j=1}^i (\al_j\sum_{c \in \mathcal{C}}\theta_{c,j}z_{c})^{n_j}, \label{eq:jointPGFNcUnconditional}
\end{align}
\end{theorem}
\begin{proof}
The proof extensively uses Theorem \ref{thm:statDist} and can be found in Appendix \ref{app:thm:jointPGFNumberCustWaitingClass}.
\end{proof}
An expression for $N$, the total number of inactive customers in the system, now follows from the fact that $\E{z^N} = \E{z^{\sum_{c \in \mathcal{C}} N^{(c)}}} = \E{\prod_{c \in \mathcal{C}}z^{N^{(c)}}}$ and $\sum_{c \in \mathcal{C}} \theta_{c,j} = \sum_{c \in \mathcal{C}} \frac{\lambda_c\ind{c \in \mathcal{U}(T_1, \ldots, T_j)}}{\lambda_{\mathcal{U}({T_1, \ldots, T_j}})}  = 1$.
\begin{corollary}\label{cor:jointPGFNumberCustWaitingGeneral}
The total number of inactive customers $N$ in the system satisfies, for $z \in \{\bar{c}\in\CC:|\bar{c}|<1\}$,
\begin{equation}
\E{z^N}=\sum_{i=0}^K \sum_{(T_1, \ldots, T_i)\in \mathcal{T}^{i}} \pi((0))\frac{\prodlambda}{\prodk} \prod_{j=1}^i\frac 1{\eta(j)}\sum_{\{n_1, \ldots, n_i\} \in \NN_0^i}\prod_{j=1}^{\sum_{k=1}^i n_k} \frac 1{\eta(i+j)} \prod_{j=1}^i (\al_jz)^{n_j}. \label{eq:jointPGFUnconditionalSimple}
\end{equation}
\end{corollary}

\subsection{Number of customers in the system}\label{sec:numberCustPresent}
We now study the number of customers present in the system, both per-type ($M^{(c)}$) and in general ($M$), by noting that these customers are comprised of inactive customers on one hand and active customers whose service is yet to be completed on the other hand. 

\begin{theorem}\label{thm:jointPGFnumberCustPresentClass}
Let $g_j$ be the type of the customer that holds token $T_j$ and define $G_{c_1, \ldots, c_i}(T_1, n_1, \ldots, T_i, n_i) := \Pro{\bigcap_{j \in \{1, \ldots, i\}} \{g_j = c_j\}\mid \st = (T_1, n_1, \ldots, T_i, n_i)}$. Then, the joint PGF of $\{M^{(c)}:c \in \mathcal{C}\}$, representing the per-class number of customers present in the system, is, for $z_c \in \{\bar{c}\in\mathcal{C}:|\bar{c}|<1$ given by
\begin{align}
\E{\prod_{c \in \mathcal{C}}z_c^{M^{(c)}}}&=\sum_{i=0}^K \sum_{(T_1, \ldots, T_i)\in \mathcal{T}^{i}} \pi((0))\frac{\prodlambda}{\prodk} \times\notag\\
&\qquad\times\prod_{j=1}^i\frac 1{\eta(j)}\sum_{\{n_1, \ldots, n_i\} \in \NN_0^i}\left(\sum_{\{c_1, \ldots, c_i\} \in \mathcal{C}^i} G_{c_1, \ldots c_i}(T_1, n_1 \ldots, T_i, n_i)\prod_{j=1}^i z_{c_j}\right) \times\notag\\
&\qquad\qquad\times\prod_{j=1}^{\sum_{k=1}^i n_k}\frac 1{\eta(i+j)} \prod_{j=1}^i (\al_j\sum_{c \in \mathcal{C}}\theta_{c,j}z_{c})^{n_j}. \label{eq:jointPGFMcUnconditional}
\end{align}
\end{theorem}
\begin{proof}
The proof is given in Appendix \ref{app:thm:jointPGFnumberCustPresentClass}.
\end{proof}
By realising that $\sum_{\{c_1, \ldots, c_i\} \in \mathcal{C}^i} G_{c_1, \ldots, c_i}(T_1, n_1, \ldots, T_i, n_i) = 1$, we again note that a PGF for the total number of customers present in the system immediately follows.
\begin{corollary}\label{cor:jointPGFnumberCustPresentGeneral}
For any $z \in \{\bar{c}\in\CC:|\bar{c}|<1\}$, the PGF of the total number of customers present in the system is given by
\begin{equation}
\E{z^M}=\sum_{i=0}^K \sum_{(T_1, \ldots, T_i)\in \mathcal{T}^{i}} \pi((0))\frac{\prodlambda}{\prodk} z^i\prod_{j=1}^i\frac 1{\eta(j)}\sum_{\{n_1, \ldots, n_i\} \in \NN_0^i}\prod_{j=1}^{\sum_{k=1}^i n_k} \frac 1{\eta(i+j)} \prod_{j=1}^i (\al_jz)^{n_j}. \label{eq:jointPGFMUnconditional}
\end{equation}
\end{corollary}
\begin{proof}
The proof follows by similar arguments as those which led to Corollary \ref{cor:jointPGFNumberCustWaitingGeneral}, together with the fact that $\sum_{\{c_1, \ldots, c_i\} \in \mathcal{C}^i} G_{c_1, \ldots, c_i}(T_1, n_1, \ldots, T_i, n_i) = 1$.
\end{proof}

\begin{remark}\label{rem:G}
A general expression for $G_{c_1, \ldots, c_i}(T_1, n_1, \ldots, T_i, n_i)$, the probability that, provided the system is in state $\st = (T_1, n_1, \ldots, T_i, n_i)$, tokens $T_1, \ldots, T_i$ are claimed by customers with types $c_1, \ldots, c_i$, respectively, seems hard to derive. For many applications, the derivation of an expression for $G_{c_1, \ldots, c_i}(T_1, n_1, \ldots, T_i, n_i)$ is, however, straightforward. For example, if the token sets $\mathcal{T}_c, c \in \mathcal{C}$, are disjoint, $G_{c_1, \ldots, c_i}(T_1, n_1, \ldots, T_i, n_i) = \ind{\cap_{j=1}^i \{T_j \in \mathcal{T}_{c_j}\}}$.
\end{remark}

\subsection{The time-till-token of customers}\label{sec:waitingTime}
We proceed to derive expressions for the time-till-token $W_c$ of a type-$c$ customer. For any $c \in \mathcal{C}$, the order in which type-$c$ customers arrive is the same as the order in which type-$c$ customers acquire a token, since tokens are always claimed by the longest waiting eligible customer. Therefore, the distributions of $N^{(c)}$ and $W^{(c)}$ satisfy the assumptions for the distributional form of Little's law to hold (cf.\ \cite{KeilsonServi}). This law dictates that, for any $s \in \{\bar{c}\in\CC:\Re(\bar{c}) > 0\}$, 
\begin{equation}\label{eq:little}
 \E{e^{-sW_c}} = \E{\left(\frac{\la_c-s}{\la_c}\right)^{N^{(c)}}}.
\end{equation}
Therefore, the following expression for $W_c$, the time-till-token of type-$c$ customers, can be obtained.

\begin{theorem}\label{thm:waiting_time} The time-till-token of a type-$c$ customer, $W_c$, satisfies, for any $s \in \{\bar{c}\in\CC:\Re(\bar{c}) > 0\}$,
\begin{align}
\E{e^{-s W_c}} &= \sum_{i=0}^K \sum_{(T_1, \ldots, T_i)\in \mathcal{T}^{i}} \pi((0))\frac{\prodlambda}{\prodk}\times \notag\\
&\qquad\times\prod_{j=1}^i\frac 1{\eta(j)}\sum_{\{n_1, \ldots, n_i\} \in \NN_0^i}\prod_{j=1}^{\sum_{k=1}^i n_k} \frac 1{\eta(i+j)} \prod_{j=1}^i \left(\al_j\left(1-\frac{s\ind{c\in \mathcal \mathcal{U}(\{T_1, \ldots, T_j\})}}{\lambda_{\mathcal{U}(\{T_1, \ldots, T_j\})}}\right)\right)^{n_j} \label{eq:LSTWc}.
\end{align}
\end{theorem}
\begin{proof}
The theorem follows by substitution of $z_d = 1$ for all $d \neq c$ in \eqref{eq:jointPGFNcUnconditional} and combining the result with \eqref{eq:little}. 
\end{proof}

%

\begin{remark}\label{sec:overallW}
From Theorem \ref{thm:waiting_time}, one can easily obtain the Laplace-Stieltjes transform (LST) of $W$ by conditioning on the type of customer, which can trivially be seen to be of type $c$ with probability $\frac{\la_c}{\la}$. In other words, $\E{e^{-sW}} = \sum_{c \in \mathcal{C}} \frac{\la_c}{\la}\E{e^{-s W}}$.
\end{remark}

\subsection{The sojourn time of customers}\label{sec:sojournTime}
In general, it is hard to derive general expressions for the sojourn time $S_c$ of a type-$c$ customer from expressions for $M^{(c)}$, as type-$c$ customers do not necessarily depart the system in the order of their arrival. Therefore, we only consider the sojourn time for instances of the model where type-$c$ customers  do depart the system in the order of arrival. This has as an advantage that the distributional form of Little's law for the quantities $M^{(c)}$ and $S^{(c)}$ holds true (cf.\ \cite{KeilsonServi}):
\begin{equation}\label{eq:littleForOverall}
 \E{e^{-sS_c}} = \E{\left(\frac{\la_c-s}{\la_c}\right)^{M^{(c)}}}
\end{equation}
for any $s \in \{\bar{c}\in\CC:\Re(\bar{c}) > 0\}$. Despite the additional assumption, the following theorem allows us to characterise the sojourn time distribution in a variety of applications in Section \ref{sec:applications}.

\begin{theorem}\label{thm:sojournTime}
If type-$c$ customers depart the system in the order of arrival, the LST of their sojourn time $S_c$ is, for $s \in \{\bar{c}\in\CC:\Re(\bar{c}) > 0\}$, given by
\begin{align}
\E{e^{-sS_c}}&=\sum_{i=0}^K \sum_{(T_1, \ldots, T_i)\in \mathcal{T}^{i}} \pi((0))\frac{\prodlambda}{\prodk} \times\notag\\
&\qquad\times\prod_{j=1}^i\frac 1{\eta(j)}\sum_{\{n_1, \ldots, n_i\} \in \NN_0^i}\left(\sum_{\{c_1, \ldots, c_i\} \in \mathcal{C}^i} G_{c_1, \ldots c_i}(T_1, n_1 \ldots, T_i, n_i)\left(\frac{\la_c-s}{\la_c}\right)^{\sum_{j=1}^i \ind{c_i=c}}\right) \times\notag\\
&\qquad\qquad\times\prod_{j=1}^{\sum_{k=1}^i n_k}\frac 1{\eta(i+j)} \prod_{j=1}^i \left(\al_j\left(1-\frac{s\ind{c\in\mathcal{U}(\{T_1, \ldots, T_j\})}}{\la_c}\right)\right)^{n_j}. \label{eq:LSTSc}
\end{align}
\end{theorem}
\begin{proof}
The proof is the same as that of Theorem \ref{thm:waiting_time}, but instead of \eqref{eq:jointPGFNcUnconditional} and \eqref{eq:little}, \eqref{eq:jointPGFMcUnconditional} and \eqref{eq:littleForOverall} are used.
\end{proof}
\begin{remark}
It is worth emphasising that, for any class $c\in\mathcal{C}$, should $|\mathcal{T}_c|=1$, the assumption that type-$c$ customers depart the system in the order of arrival is always valid. If $\mathcal{T}_c = \{t\}$, and it holds moreover that token $t$ can only be claimed by type-$c$ customers (i.e.\ $\mathcal{C}_t = \{c\}$), the PGF of the sojourn time distribution simplifies:
\begin{align}
\E{e^{-sS_c}}&=\sum_{i=0}^K \sum_{(T_1, \ldots, T_i)\in \mathcal{T}^{i}} \pi((0))\frac{\prodlambda}{\prodk} \frac{\la_c-s\ind{t\in\{T_1, \ldots, T_i\}}}{\la_c}\times\notag\\
&\qquad\times\prod_{j=1}^i\frac 1{\eta(j)}\sum_{\{n_1, \ldots, n_i\} \in \NN_0^i}\prod_{j=1}^{\sum_{k=1}^i n_k}\frac 1{\eta(i+j)} \prod_{j=1}^i \left(\al_j\left(1-\frac{s\ind{t\in\{T_1, \ldots, T_j\}}}{\la_c}\right)\right)^{n_j}. \label{eq:LSTScSimple}
\end{align}
We will use these simplification amply in Section \ref{sec:applications}.
\end{remark}
\begin{remark}
In Remark \ref{rem:closedFormPiZero}, we saw that the stationary distribution allows for a closed-form expression when $\eta(\cdot) = 1$. In Appendix \ref{app:eta1}, simplified expressions for several performance measures studied in this section are given, which follow by substitution of $\eta(\cdot)=1$. These expressions show that this parameter setting also leads to closed-form expressions for the performance measures. In fact, it follows from inversion of these expressions that for all $ c \in \mathcal{C}$, the quantities $N_{c}$ and $N$ can be interpreted as a weighted convolution of geometric random variables. Likewise, the time-till-token of a customer (either of a particular type or of an arbitrary type by Remark \ref{sec:overallW}), can be interpreted as a weighted convolution of exponential random variables. 
\end{remark}

\section{Generalisation of two existing classes of models}
\label{sec:generalisation}
In this section, we show that both the multi-type customer and server model (cf.\ \cite{Visschers12}) as well as the order-independent queue (cf.\ \cite{Krzesinski11}) can be seen as special cases of the token-based central queue as analysed in this paper. The results are summarised in the Venn diagram presented in Figure~\ref{fig:venn}. This figure also categorises applications of token-based central queues, which we will consider in Section \ref{sec:applications}.

\begin{figure}[t]
\centering
\includegraphics[width=0.60\textwidth]{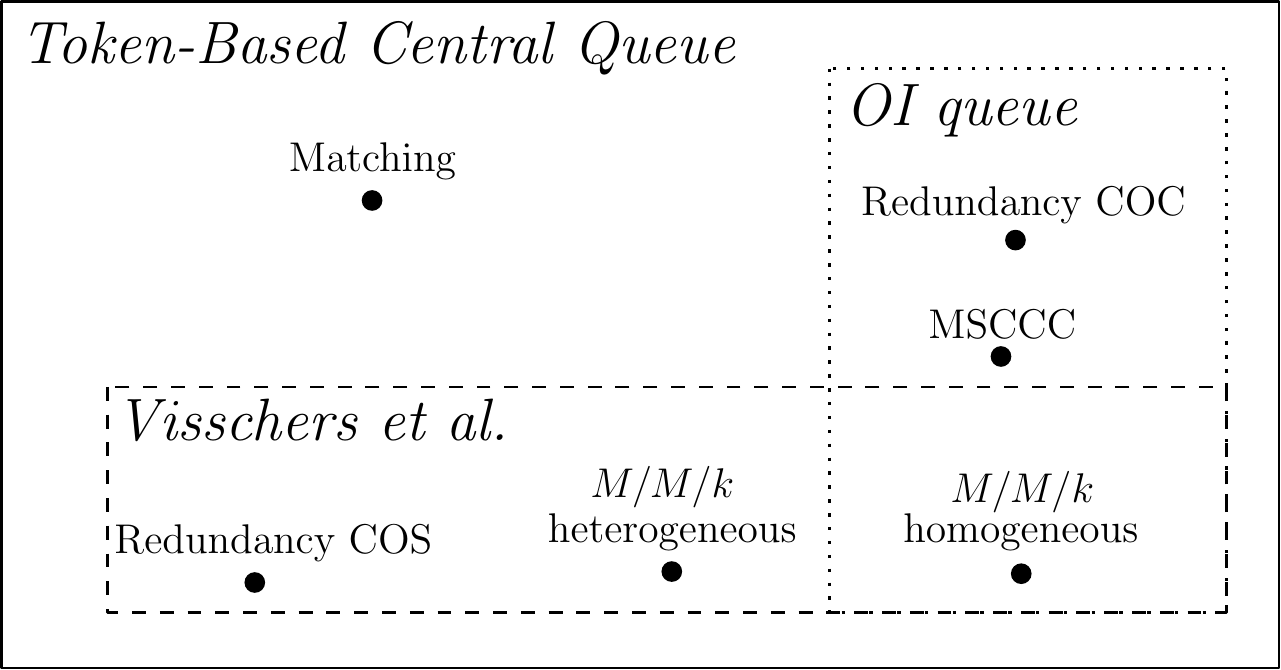}
\caption{A classification of token-based central queues}
\label{fig:venn}
\end{figure}

\subsection{Multi-type customer and server model}\label{sec:Visschers}
In the multi-type customer and server model of~\cite{Visschers12}, customers of type~$c$ arrive at the system according to a Poisson process  with rates~$\lambda_c$ and have an exponentially distributed service requirement with rate~1.  
There are $K$ machines and machine~$i$ works at rate~$\mu_i$. Each customer type has a set of compatible machines it can be served at. Whenever a machine becomes idle, it takes the earliest arrived customer in the queue that it can process. An arriving customer that finds more than one compatible server idle is assigned to one of the servers according to a random assignment rule that satisfies a certain assignment condition. 

The model of \cite{Visschers12} is  a special case of our token-based central queue. There are $K$ tokens, where each token represents a machine. Taking $\eta(\cdot)=1$ and $s_j(T_1,\ldots, T_i)=\mu_j$ in our token-based central queue, we directly retrieve the multi-type customer and server model. The assignment condition of \cite{Visschers12} coincides with Condition \ref{cdn:assignment} and it is immediately seen that the order-independent condition (Condition~\ref{cdn:OI}) is also satisfied. We hence retrieve that the stationary distribution is of product-form type, as was shown in~\cite{Visschers12}.

\subsection{The OI queue}\label{sec:OI}
The  order-independent (OI) queue was first described in~\cite{Krzesinski11}. This model consists of a single central queue where arriving multiclass customers wait
in a FIFO order. 
Customers of type $i$
arrive according to a Poisson 
process with rate $\lambda_i$ and have an exponentially distributed service requirement with rate 1.
The generic state descriptor as considered in~\cite{Krzesinski11} is 
$\stOI = (c_1, \ldots, c_n)$, where $n$ is the number of  customers in the system and $c_j$ denotes the type of the
$j^{th}$ customer in the central queue. 
Let $\stsp^{(OI)}$ denote the corresponding state space.
For a given state $\stOI\in \stsp^{(OI)}$, let ${\mu}_j^{(OI)}(\stOI)$ 
denote the departure rate associated with the $j$-th customer.
In an OI queue it is assumed that the following condition holds.
\begin{condition}
\label{cond:X}
In a given state $\stOI=(c_1,\ldots, c_n)$, each of the rates ${\mu}_j^{(OI)}(\stOI)$, $j=1,\ldots, n$,  can be written as
\begin{equation}
\label{eq:OIeq}
{\mu}_{j}^{(OI)}(\stOI) = \eta^{(OI)}(n) s_j^{(OI)}(c_1, \ldots, c_n),
\end{equation}
where
\begin{enumerate}
\item  $s_j^{(OI)}(c_1, \ldots, c_n) = s_j^{(OI)}(c_1, \ldots, c_j)$ 
for any $1 \leq j \leq i$
\item $k^{(OI)}(c_1, \ldots, c_n) := \sum_{j=1}^{n} s_j^{(OI)}(c_1, \ldots, c_j)$ is 
independent of any permutation of $(c_1, \ldots, c_n)$ and
\item $\eta^{(OI)}(n) > 0$ for $n > 0.$
\end{enumerate}
\end{condition}




We see a close similarity with the order-independent condition as stated in Condition~\ref{cdn:OI}.  In the following results, we will clarify the connection between the two modelling frameworks. We will use the notion of indistinguishable tokens as introduced in Section \ref{sec:L}, as well as the state descriptor of the form $\st^{(L)} = (L_1, n_1, \ldots, L_i, n_i)$ and the corresponding steady-state distribution (cf.\ \eqref{eq:statdistL}). We start out with a preparatory lemma.

\begin{lemma}
\label{lem:tau}
For any token-based central queue where each token set $\mathcal{T}_c$, $c \in \mathcal{C}$, consists of indistinguishable tokens, there exists a function $\mathcal{\tau}: \stsp^{(OI)} \rightarrow \stsp^{(L)}$, where $\tau(\stOI)\in \stsp^{(L)}$ denotes the unique state $(L_1,n_1,\ldots, L_i,n_i)$  corresponding  to   the  state  $\stOI \in \stsp^{(OI)}$. 
\end{lemma}

\begin{proof}
Since each class has one token label it can select from, this guarantees that there is no ambiguity about how the tokens are distributed among the customers. By keeping track of the order of arrival and the token labels allotted to the customers, one can construct the unique state~$\st^{(L)}=(L_1, n_1, \ldots, L_i, n_i)$ corresponding to~$\stOI = (c_1, \ldots, c_n)$, that is, the function $\mathcal{\tau}(\cdot)$ as stated in the lemma exists. The quantity $n_i$ represents customers without a claimed token.
\end{proof}

This lemma allows us to prove the following theorem, which exposes the connection between OI queues and token-based central queues.

\begin{theorem} 
\label{lem:keep}
For a given model, the following are equivalent:
\begin{itemize}
 \item[(1)]  the model fits in the OI queue framework;
 \item[(2)] the model can be seen as  a token-based central queue where the token sets associated with each of the classes each consist of indistinguishable tokens.
 \end{itemize}
\end{theorem}
\begin{proof}
$(1)\to(2)$: Given a model that fits in the OI queue framework, we define 
the token set of customer type~$c$, $\mathcal{T}_c$, as a set containing an infinite number of indistinguishable tokens with label $c$. Since each customer type is then represented by its own token label, the state $\st^{(L)} = (L_1, \ldots, L_i)$ coincides with that of $\stOI = (c_1, \ldots, c_i)$. Since the model satisfies Condition~\ref{cond:X}, it is direct that the token sets $\mathcal{T}_c$  define a token-based central queue.
 
$(2)\to(1)$:
For a token-based central queue where the tokens in each $\mathcal{T}_c$ are indistinguishable, Lemma~\ref{lem:tau} implies the  existence of a function $\tau$
that transforms a state $\stOI$ into the corresponding state $\st^{(L)}$.  Since the departure rates of the customers in the token-based central queue are a function of $\st$, or equivalently of $\st^{(L)}$, one can equivalently define departure rate functions  $\mu_j^{(OI)}( \cdot )$ for the $j$-th  active customer in the system as a function of $\tau(\stOI)$.  These departure rate functions satisfy Condition~\ref{cond:X}, so that (1) of Theorem~\ref{lem:keep} follows.  See Appendix~\ref{app:lem:keep} for a detailed proof.
\end{proof}

The above theorem states that, given some model, one can interpret it as an OI queue if and only if the model can be interpreted as a token-based central queue where the token set of each customer type contains indistinguishable tokens. It is important to note the difference in the two state representations $\st^{(OI)}$ and ~$\st^{(L)}$: the types of all the customers are known in the OI queue, while only the customer types associated with the \textit{active customers} can be known in our token-based representation. However, this sacrifice of detail leads to a richer class of models, as the above results show that the class of token-based central queues can handle a larger set of applications than the class of OI queues. 

For both representations, a product-form solution for the steady-state distribution exists; see \eqref{eq:statdistL} for our representation and the corollary below for the OI representation.

\begin{corollary}
\label{cor}
If the model fits in the OI queue framework, the steady-state distribution in terms of the OI state descriptor,  denoted by $\pi^{(OI)}(\stOI)$,  is given by
\begin{equation}
\label{eq:OIstat}\pi^{(OI)}(\stOI) = \pi^{(OI)}((0)) \prod_{i=1}^n \frac{\lambda_{c_i}}{\eta^{(OI)}(i)k^{(OI)}(c_1, \ldots, c_i)},
\end{equation}
 as was derived in~\cite{Krzesinski11}.
\end{corollary}

\begin{proof}
From Theorem~\ref{lem:keep} together with \eqref{eq:statdistL}, we can recover the steady-state distribution in terms of the OI state descriptor. See Appendix~\ref{app:cor} for the full proof. 
\end{proof}

\section{Applications}\label{sec:applications}
In this section, we treat a few applications that can be analysed by interpreting them as token-based central queues. For illustrative purposes, we show in Section 7.1 and Section 7.2 how to analyse an M/M/K queue with heterogeneous service rates and an MSCCC queue using our model. The first of these models fits in the framework of \cite{Visschers12}, while the second is an OI queue. Then, we apply the results on the COS and the COC redundancy models in Section \ref{sec:redundancy}, and in the process obtain new expressions of several performance measures for these models. While the COS model can be interpreted as an instance of the model of \cite{Visschers12}, the COC redundancy model can be interpreted as an OI queue. Finally, in Section \ref{sec:matching}, we discuss matching models, which are neither OI queues nor fit the framework of \cite{Visschers12}. These applications have been categorised in Figure \ref{fig:venn}. It is worth emphasising that, especially for the latter models, this section includes results on performance measures that have not been derived in the literature before.

\subsection{M/M/K queue with heterogeneous service rates}\label{sec:heteroMMK}
We first regard the M/M/K queue with heterogeneous service rates. This is a single-class queue served by $K$ servers labeled $t_1, \ldots, t_K$, to which customers arrive according to a Poisson process with rate $\lambda$. Upon arrival, the customer is assigned any available server uniformly at random. In case there are no available servers, the customer waits in the queue which is processed by the servers in order of arrival. A customer who is served by server $t_i$ has an exponentially distributed service time with parameter $\mu(t_i)$. We denote the sum of the service rates by $\mu = \sum_{i=1}^K \mu(t_i)$. This system was studied in detail by \cite{Gumbel}, and by applying the token-based framework, much of the performance analysis from that paper can be recovered.

\subsubsection{Choice of model parameters}
By introducing a token for every server, one can interpret the M/M/K queue with heterogeneous service rates as a token-based central queue. As the servers are identified by the tokens, we label the tokens as $t_1, \ldots, t_K$ as well. To receive service, a customer must hold any of the $K$ tokens. The assignment rule of this system requires that upon arrival of a customer, when tokens $T_1, \ldots, T_{j-1} \in \mathcal{T}$ are unavailable, the customer claims any of the other tokens $t \in \mathcal{T} \backslash\{T_1, \ldots, T_{j-1}\}$ with uniform probability. In other words, for $j = 1, \ldots, K$ and $t \in \mathcal{T} \backslash \{T_1, \ldots, T_{j-1}\}$, we have $\lambda_{t}(T_1, \ldots, T_{j-1}) = \frac{\lambda}{K-j+1}$. Since there is only one customer class, $\lambda_{\mathcal{U}(\mathcal{T})} = \lambda$ and $\lambda_{\mathcal{U}(\mathcal{R})}=0$ for any strict subset $\mathcal{R}$ of $\mathcal{T}$.  Condition \ref{cdn:assignment} is now satisfied. To match the departure rates of the M/M/K queue, we adopt the parameters $\eta(j) = 1$ for all $j \in \NN$, $s_j(T_1, \ldots, T_i) = \mu(T_j)$ and $k(T_1, \ldots, T_i) = \sum_{j=1}^i \mu(T_j)$, which satisfy Condition \ref{cdn:OI}. Finally, since the system has a single customer class ($\mathcal{C} = \{c\})$, we have that $G_{c, \ldots, c}(T_1, n_1, \ldots, T_i, n_i) = 1$.

\subsubsection{Performance analysis}\label{sec:heteroMMKPerfEval}
Equation \eqref{eq:statDist} leads to the following distribution. For any $\st = (T_1, T_2, T_3, \ldots, T_k, n_k)\in \stsp$, we have that
\begin{equation}\label{eq:statDistMMK}
\pi(\st) = \pi((0)) \frac{\prod_{j=1}^i\frac{\lambda}{K-j+1}}{\prod_{j=1}^i \sum_{l=1}^j \mu(T_l)}\left(\frac{\lambda}{\mu}\right)^{\ind{i=K}n_K} = \pi((0)) \frac{\lambda^i(K-i)!}{K!\prod_{j=1}^i \sum_{l=1}^j \mu(T_l)}\left(\frac{\lambda}{\mu}\right)^{\ind{i=K}n_K},
\end{equation}
while $\pi(\st)=0$ for all other states. The term $\pi((0))$ is a normalising constant. It is possible, however, to drop the ordering of the tokens from the state descriptor, while the system remains Markovian. Instead, states of the form $(n, \mathcal{R})$ can be introduced, where $n$ is the number of waiting customers and $\mathcal{R}$ represents the (orderless) set of servers/tokens in service. By aggregation of states, we obtain 
\begin{equation*}
\pi(n, \mathcal{R}) = \frac{\pi((0))\lambda^{|\mathcal{R}|}(K-|\mathcal{R}|)!}{K!}\left(\frac{\lambda}{\mu}\right)^{n}\sum_{(T_1, \ldots, T_{|\mathcal{R}|})\in\underline{\mathcal{R}}}\frac{1}{\prod_{j=1}^{|\mathcal{R}|}\sum_{l=1}^j \mu(T_{l})},
\end{equation*}
where $\underline{\mathcal{R}}$ is the set of all possible permutations of the tokens in $\mathcal{R}$. It can be deduced from Corollary \ref{cor:jointPGFNumberCustWaitingGeneral} that  the stationary number of inactive customers are geometrically ($1-\frac{\lambda}{\mu}$) with probability $\left(\frac{\lambda}{\mu}\right)^K$, and equals zero otherwise. Likewise, Theorem \ref{thm:waiting_time} leads to the fact that the stationary waiting time is exponentially ($\mu-\lambda$) distributed with probability $\left(\frac{\lambda}{\mu}\right)^K$ and equals zero otherwise.

The (PGF of the) number of customers in the system is by virtue of Corollary \ref{cor:jointPGFnumberCustPresentGeneral} given by
\begin{equation*}
\E{z^M} = \pi((0)) \sum_{i=0}^K\sum_{(T_1, \ldots, T_i)\in\mathcal{T}^{i}}\frac{\lambda^i(K-i)!z^i}{K!\prod_{j=1}^i\sum_{l=1}^j \mu(T_l)}\left(\frac{\lambda}{\mu}\right)^{\ind{i=K}n_K},
\end{equation*}
for any $z \in\{\bar{c}\in\CC:|\bar{c}|<1\}$. From this, it follows directly that also $M$ has a geometric law with probability $\left(\frac{\lambda}{\mu}\right)^K$. Finally, to obtain expressions for the sojourn time, the results of Section \ref{sec:sojournTime} do not apply, since customers of equal types do not necessarily leave the system in the order of their arrival. Instead, through a PASTA-argument and by conditioning on the server that an arriving customer will be served by, the following LST for the sojourn time $S$ can be derived for any $s \in \{\bar{c}\in\CC:\Re(\bar{c}) > 0\}$:
\begin{align*}
\E{e^{-sS}} &= \left(\sum_{\mathcal{R}\subset\mathcal{T}:\mathcal{T}\backslash\mathcal{R}\neq\emptyset} \pi(0, \mathcal{R}) \frac{\lambda}{K-|\mathcal{R}|}\sum_{T\in\mathcal{T}\backslash\mathcal{R}}\frac{\mu(T)}{\mu(T)+s}\right) + \\
&\qquad\qquad + \sum_{n=0}^\infty\pi(n, \mathcal{T}) \E{e^{-sW} \mid W>0}\sum_{T\in \mathcal{T}} \frac{\mu(T)}{\mu}\frac{\mu(T)}{\mu(T)+s} \\
&=\left(\sum_{\mathcal{R}\subset\mathcal{T}:\mathcal{T}\backslash\mathcal{R}\neq\emptyset} \pi(0, \mathcal{R}) \frac{\lambda}{K-|\mathcal{R}|}\sum_{T\in\mathcal{T}\backslash\mathcal{R}}\frac{\mu(T)}{\mu(T)+s}\right) + \left(\frac{\lambda}{\mu}\right)^K\frac{\mu-\lambda}{\mu-\lambda+s}\sum_{T\in \mathcal{T}} \frac{\mu(T)}{\mu}\frac{\mu(T)}{\mu(T)+s},
\end{align*}
where terms between brackets represent the case where an arriving customer is immediately served.

\begin{remark}
The M/M/K queue with heterogeneous service rates is not an OI queue. This follows since the service rate is driven by the server/token and not by the class of the customer which it is serving. When the service rates of the servers are equal, however, we obtain a conventional Erlang C model, which does fit in the OI queue framework. 
\end{remark}

\begin{remark}
For this system, it would make sense to introduce an assignment rule so that an arriving customer chooses the server with the highest service rate. However, such an assignment rule would violate Condition \ref{cdn:assignment}.
\end{remark}

\subsection{The MSCCC queue}\label{sec:MSCCC}
We now illustrate an application where results on the sojourn time in Section \ref{sec:sojournTime} can be applied. We study the Multi-server Station with Concurrent Classes of Customers (MSCCC) queue. Studied in \cite{Boudec} and \cite{Crosby90}, this queue contains multiple servers and multiple classes of customers, where at most one customer of any type can be in service. More particularly, the MSCCC queue consists of $k$ identical servers serving customers at unit rate. Customers of type $c_l$, $l\in \mathbb{N}$, arrive according to a Poisson process with rate $\lambda_{c_l}$ and have exponential($\mu$) service requirements. Upon arrival, when a server is available and no other customer of his/her type is in service, the customer will go into service at an arbitrary free server. When no server is available or another customer of its type is already in service, the customer waits in line. When a server becomes available, it scans the queue from the front for the first customer eligible for service (i.e. the longest waiting customer of a type that is not in service at the moment). Through the token-based framework, we derive expressions for relevant performance measures of the MSCCC queue, which to the best of the authors' knowledge has not been done before. 

\subsubsection{Choice of model parameters}
To model the MSCCC queue using the token-based representation, we introduce for every customer type $c_l$ a token $t_l$, which is dedicated to type-$c_l$ customers. Thus, token $t_l$ will always be held by the oldest type-$c_l$ customer in the system if there is any, otherwise it is available. Given the one-to-one correspondence between customer types and tokens, we will henceforth refer to the type of a customer by its corresponding token. For example, we refer to the arrival rate of a type-$c_l$ customer with $\lambda_{t_l}$. Then, it holds that $\lambda_{t_l}(T_1, \ldots, T_i) = \lambda_{t_l}$ in case $t_l \notin \{T_1, \ldots, T_i\}$. It follows trivially that $\lambda_{\mathcal{U}(T_1, \ldots, T_i)} = \sum_{j=1}^i \lambda_{T_j}$. The departure rates can be characterised by choosing $\eta(j) = 1$, $s_j(T_1, \ldots, T_i) = \mu\ind{j\le k}$ and $k(T_1, \ldots, T_i) = \min\{i,k\}\mu$ for any combination of active tokens $(T_1, \ldots, T_i)$ and $j = 1, \ldots, i$. Note that these parameter settings satisfy Conditions \ref{cdn:assignment} and \ref{cdn:OI}. Also, it is trivial to note that $G_{c_1, \ldots, c_i}(T_1, n_1, \ldots, T_i, n_i) = \ind{\bigcap_{j=1}^i \mathcal{C}_{T_j}) = \{c_j\}}$.

\subsubsection{Performance analysis}
Treating the MSCCC queue as a token-based central queue with the model parameters outlined above, the stationary distribution now follows from \eqref{eq:statDist}. In particular, for any state $\st = (T_1, n_1, \ldots, T_i, n_i) \in \stsp$,
\begin{align*}
\pi(\st) &= \pi((0))\frac{\prod_{j=1}^i \lambda_{T_j}}{\left(\prod_{j=k+1}^i k\mu\right)\left(\prod_{j=1}^{\min(i,k)}j\mu\right)} \prod_{j=1}^i\left(\frac{\sum_{l=1}^j \lambda_{T_l}}{\min(j,k)\mu}\right)^{n_j} = \pi((0)) \frac{\prod_{j=1}^i \lambda_{T_j}\left(\frac{\sum_{l=1}^j \lambda_{T_l}}{\min(j,k)\mu}\right)^{n_j}}{\mu^i\min(i,k)!k^{\max(i-k, 0)}},
\end{align*}
where $\pi((0)) = \left(\sum_{\st \in \stsp} \frac{\prod_{j=1}^i \lambda_{T_j}}{\mu^i\min(i,k)!k^{\max(i-k, 0)}}\right)^{-1}$ is a normalising constant. It is worth emphasising that in this system, $N^{(c_l)}$ cannot be interpreted as the number of waiting type-$c_l$ customers, since a customer can hold a token while not receiving service. It makes more sense to compute the joint number of per-type customers in the system. After substitution of the model parameters in \eqref{eq:jointPGFMcUnconditional} and simplification of the result, we obtain, for $z_c \in\{\bar{c}\in\CC:|\bar{c}|<1\}$,
\begin{equation*}
\E{\prod_{c \in \mathcal{C}}z_c^{M^{(c)}}} = \sum_{i=0}^K \sum_{(T_1, \ldots, T_i) \in \mathcal{T}^{i}} \pi((0)) \prod_{j=1}^i \frac{\lambda_{T_j}z_{T_j}}{k(T_1, \ldots, T_j)-\sum_{l=1}^j \lambda_{T_l}z_{T_l}}.
\end{equation*}
Similarly, the time-till-token of a type-$c_l$ customer does not reflect its waiting time. The waiting-time distribution of a type-$c_l$ customer can instead be recovered through its sojourn time distribution, as the sojourn time distribution in this case is a convolution of the waiting-time distribution and the exponential($\mu$) service-time distribution. The MSCCC queue satisfies the condition that same-type customers depart the system in the order they arrive, and hence, the sojourn time can be computed using Theorem \ref{thm:sojournTime}. For the MSCCC queue, the LST of the sojourn time  distribution of a type-$c_l$ customer in \eqref{eq:LSTScSimple} simplifies to
\begin{equation*}
\E{e^{-sS_{c_l}}} = \sum_{i=0}^K \sum_{(T_1, \ldots, T_i)\in\mathcal{T}^{i}} \pi((0)) \prod_{j=1}^i \frac{\lambda_{T_j}-s\ind{T_j=t_l}}{\min(j,k)\mu-\sum_{l=1}^j \lambda_{T_l}+s\ind{t_l \in \{T_1, \ldots, T_j\}}},
\end{equation*}
where $s \in \{\bar{c}\in\CC:\Re(\bar{c}) > 0\}$. Finally, the LST of the waiting-time distribution of a type-$c_l$ customer is given by $(\mu+s)\E{e^{-sS_{c_l}}}/\mu$.

\begin{remark}
For ease of notation, we assumed that each of the customer classes share the same service requirement distribution. However, the case where customer type have mutually different exponential service requirement distributions can also be modelled as a token-based central queue. Furthermore, extensions of this queue have been studied in the literature, where for every subset of customer types, there is a maximum defined of how many customers with those types can be in service at any particular point in time. Also this extension falls in the token-based framework when choosing the model parameters carefully. However, its (per-type) stationary sojourn time distribution cannot be derived by the methods derived in this paper, as the assumptions required to use the distributional form of Little's law are violated.
\end{remark}

\subsection{Redundancy models}\label{sec:redundancy}
A timely application of the token-based central queue is given by redundancy systems. The study of redundancy systems has gained momentum recently, as mentioned in Section \ref{sec:related_work}. One example of such a system is the redundancy-$d$ cancel-on-start (COS) model studied in \cite{Ayesta18}. This model constitutes a system with $K$ single-server FCFS queues and homogeneous servers providing service at equal speed. Customers arrive according to a Poisson process at rate $\lambda$. Upon arrival, the customers choose at random $d$ out of $K$ queues, and to each of those queues, a copy of the customer is sent, each copy having its own independent, exponentially($\mu$) distributed service requirement. Under COS, once service on any of these copies has started, all the other copies of the same customer are removed from the system, and only the sole remaining copy will be serviced. In case service on multiple copies could start at the same time since a customer find multiple of its $d$ eligible servers idle upon arrival, it will apply a uniform assignment rule. That is, it will uniformly at random select an idle server where service on a copy will be continued, and copies sent to other idle servers are instantly terminated. It is worth emphasising that it was shown by \cite{Ayesta18} that this model fits the framework of \cite{Visschers12}, and hence can be interpreted as a token-based central queue. In Section \ref{sec:cos}, we recall how the redundancy-$d$ COS model can be interpreted as a token-based central queue, after which we complement the results of \cite{Ayesta18} by deriving novel expressions for performance measures such as the customers' waiting-time distribution. Then, in Section \ref{sec:coc}, we study a variant of this model, namely the redundancy-$d$ cancel-on-completion (COC) model. This model shares the same characteristics as the COS-model, with the exception that redundant customer copies will now only be removed once any of the copies has \emph{completed} service. Therefore, it is now possible that multiple copies of the same customer are in service at the same time. In \cite{Gardner17}, it is shown that the steady-state distribution for that model allows for a product-form solution when using a different state descriptor than ours. Furthermore, they analyse the sojourn time distribution in limiting regimes and derive the mean sojourn time for the general case. It was shown in \cite{Ayesta18} that the COC-variant of the redundancy-$d$ model can be interpreted as a special version of a token-based central queue, and as such the stationary distribution of the model using our state descriptor also leads to a product-form stationary distribution. In this section, using results from Section \ref{sec:perfEval}, we supplement the analysis of \cite{Gardner17} by giving a characterisation of the complete distribution of the sojourn time. We also give expressions for other performance measures. It should be noted that, although we will view the COS model and the COC model as token-based central queues, the COS and COC model both actually consists of $K$ parallel queues.

\subsubsection{The redundancy-$d$ COS model}\label{sec:cos}
To analyse the COS model, we first present it as an instance of a token-based central queue by choosing adequate model parameters. 

\paragraph{Choice of model parameters}\mbox{ }\\
To interpret the COS-model as a token-based central queue, we introduce a token set $\mathcal{T} = \{t_1, \ldots, t_K\}$, where token $t_i$ has a one-to-one correspondence to the $i$-th of the $K$ servers. We also introduce customer types that correspond to the set of servers/tokens an arriving customer replicates to. Thus, equal-type customers send copies to the same $d$ out of $K$ servers. As a consequence, there are $\binom{K}{d}$ customer types, labeled $c_1, \ldots, c_{\binom{K}{d}}$, which are ordered lexicographically. When a token is said to be claimed by a customer, the customer is taken into service by the server corresponding to the claimed token, so that copies sent to other servers are cancelled. It follows from this setting that if customers of type $c$ send copies to servers in the set $\mathcal{R}\subset\mathcal{T}$, customers of type $c$ are only able to claim the tokens in the set $\mathcal{R}$, i.e. $\mathcal{T}_c = \mathcal{R}$.

Since an arriving customer is of any of the $\binom{K}{d}$ types with uniform probability, we have $\lambda_{c_i} = \frac{\lambda}{\binom{K}{d}}$. Deriving $\lambda_{t}({T_1, \ldots, T_{j-1}})$ is more intricate. Suppose that an arriving customer finds $a$ tokens available (or servers idle). It will then immediately claim any one of them with probability $\frac{1}{a}$. The uniform assignment rule also dictates that when tokens $(T_1, \ldots, T_{j-1})$ are active, this means that there are $\binom{K-j}{a-1}\binom{j-1}{d-a}$ customer types of which an arriving customer, upon arrival, would find a tagged token $t$ among the $a$ available tokens that it could immediately claim. That is, $t$ is one of the eligible available tokens, there are $a-1$ others out of the $K-j$ available tokens ($\binom{K-j}{a-1}$ possibilities) and the remaining $d-a$ out of the $d$ eligible tokens are among $T_1, \ldots, T_{j-1}$ ($\binom{j-1}{d-a}$ possibilities). Combining these observations and adhering to the standard convention that $\binom{m}{n} = 0$ for $0\le m <n$, we have 
\begin{equation*}
\lambda_{t}(T_1, \ldots, T_{j-1}) = \sum_{a=1}^{\min\{K-j+1, d\}} \frac{\lambda}{\binom{k}{d}}\frac{1}{a}\binom{K-j}{a-1}\binom{j-1}{d-a}
\end{equation*}
for any $(T_1, \ldots,T_{j-1})\in \mathcal{T}^{j-1}$ and any $t \in \mathcal{T}\backslash\{T_1, \ldots, T_{j-1}\}$. Due to symmetry, it is immediate that Condition \ref{cdn:assignment} is satisfied. We also reason that $\lambda_{\mathcal{U}(\{T_1, \ldots, T_i\})} = \frac{\lambda\binom{i}{d}}{\binom{K}{d}}$, since out of the $\binom{K}{d}$ customer types, there are $\binom{i}{d}$ that replicate to $d$ queues corresponding to servers/tokens in the set $\{T_1, \ldots, T_i\}$. 

The selection of departure rate parameters is significantly easier. When a copy of a customer starts service (i.e., claims a token), its departure rate from the system equals $\mu$. This is reflected by choosing $\eta(j)=1$ for all $j \in \mathbb{N}$ and $s_j(T_1, \ldots, T_i) = \mu$ for all possible sets $(T_1, \ldots, T_i)$ of $i$ tokens, so that $k(T_1, \ldots, T_i) = i\mu$. Under the parameter settings just introduced, the token-based central queue has the exact same behaviour as a redundancy-$d$ model. Furthermore, by probabilistic reasoning, we have that $G_{c_1, \ldots, c_i}(T_1, n_1, \ldots, T_i, n_i) = \frac{\lambda\ind{\bigcap_{j=1}^i \{T_j \in \mathcal{T}_{c_j}\}}}{\binom{k-1}{d-1}}$, since any server/token can be selected by $\binom{k-1}{d-1}$ customer types.

\paragraph{Performance analysis}\mbox{ }\\
With the model parameters selected as above, we immediately obtain the stationary distribution of the redundancy-$d$ COS model. For any $\st = (T_1, n_1, \ldots, T_i, n_i)\in \stsp$, we have that
\begin{equation*}
\pi(\st) = \pi((0)) \frac{\prod_{j=1}^i \sum_{a=1}^{\min\{K-j+1, d\}} \frac{\lambda}{\binom{k}{d}}\frac{1}{a}\binom{K-j}{a-1}\binom{j-1}{d-a}}{i!\mu^i}\prod_{j=d}^i \left(\frac{\lambda\binom{j}{d}}{j\mu\binom{k}{d}}\right)^{n_j},
\end{equation*}
where $\pi((0))$ is a normalising constant. Note that this expression equals the stationary distribution found for the redundancy-$d$ COS model in \cite[Proposition 1]{Ayesta18}, as expected. Next, it can be deduced from Corollaries \ref{cor:jointPGFNumberCustWaitingGeneral} and \ref{cor:jointPGFnumberCustPresentGeneral} that the number of waiting customers and the total number of customers in the system are, for $z \in\{\bar{c}\in\CC:|\bar{c}|<1\}$ characterised by
\begin{equation*}
\E{z^N} = \sum_{i=0}^K\sum_{(T_1, \ldots, T_i)\in\mathcal{T}^{i}}\pi((0)) \prod_{j=1}^i\frac{\sum_{a=1}^{\min\{K-j+1, d\}} \frac{\lambda}{\binom{k}{d}}\frac{1}{a}\binom{K-j}{a-1}\binom{j-1}{d-a}}{j\mu - \lambda\frac{\binom{j}{d}z}{\binom{k}{d}}}
\end{equation*}
and
\begin{equation*}
\E{z^M} = \sum_{i=0}^K\sum_{(T_1, \ldots, T_i)\in\mathcal{T}^{i}}\pi((0))z^i \prod_{j=1}^i\frac{\sum_{a=1}^{\min\{K-j+1, d\}} \frac{\lambda}{\binom{k}{d}}\frac{1}{a}\binom{K-j}{a-1}\binom{j-1}{d-a}}{j\mu - \lambda\frac{\binom{j}{d}z}{\binom{k}{d}}}.
\end{equation*}
Likewise, since the time-till-token of this system coincides with the time until a customer receives service, Theorem \ref{thm:waiting_time} gives us the (PGF of the distribution of the) customers' waiting time. Exploiting symmetry, this leads, for $s \in \{\bar{c}\in\CC:\Re(\bar{c}) > 0\}$, to
\begin{equation*}
\E{e^{-sW}} = \E{e^{-sW_{c_1}}} = \sum_{i=0}^K\sum_{(T_1, \ldots, T_i)\in\mathcal{T}^{i}}\pi((0)) \prod_{j=1}^i\frac{\sum_{a=1}^{\min\{K-j+1, d\}} \frac{\lambda}{\binom{k}{d}}\frac{1}{a}\binom{K-j}{a-1}\binom{j-1}{d-a}}{j\mu - \lambda\frac{\binom{j}{d}}{\binom{k}{d}}+s\ind{j\ge d}}.
\end{equation*}
As for the sojourn time of customers, it is worth noting that Theorem \ref{thm:sojournTime} does not apply to the COS model, since same-type customers do not claim a token in the order of arrival. Hence, the distributional form of Little's law does not hold. However, since each customer's service time is independent of its waiting time, we have that $\E{e^{-sS}} = \frac{\mu\E{e^{-sW}}}{\mu+s}$. 
\begin{remark}
As pointed out in \cite{Ayesta18} and \cite{HellemansVanHoudt}, the redundancy-$d$ COS model is equivalent to a Join-the-Shortest-Work queue, where an arriving customer opts to join the least-loaded of $d$ random queues. As a consequence, performance measures for that queue are therefore known as well.
\end{remark}

\begin{remark}
Note that for the redundancy-$d$ COS model, it is not necessary to assume that service requirements of customers copies are independent of one another. This is due to the fact that only a single copy of a customer will ever receive service, and the service requirements of the copies do not influence which copy eventually gets served. In the COC-model that we will study next, this does not hold true, so that the assumption of independent service requirements is essential.
\end{remark}

\subsubsection{The redundancy-$d$ COC model}\label{sec:coc}
Recall that the COC model differs from the COS model in that redundant copies are now only cancelled once any of the copies has \emph{completed} service. Although this difference seems minor, the performance measures are affected significantly. To allow the COC model to be interpreted as a token-based central queue, the model parameters need to be interpreted in a different way.

\paragraph{Choice of model parameters}\label{sec:cocModelParams}\mbox{ }\\
As we did for the COS model, we introduce a customer class for every choice of $d$ out of $K$ servers an arriving customer replicates to, so that there are $\binom{K}{d}$ customer classes in total. However, we do not associate tokens with servers, but with customer classes, much like the MSCCC queue treated in Section \ref{sec:MSCCC}. This is possible, since in a COC model only the oldest of the customers of any type can receive actual service by a server. This is a direct consequence of the queues of the redundancy-$d$ model being served in the order of arrival. Thus, we now introduce a token set $\mathcal{T} = \{t_1, \ldots, t_{\binom{K}{d}}\}$, where $t_i$ corresponds to the $i$-th of the $\binom{K}{d}$ customer classes. Since every customer class has its dedicated token, we have that $\lambda_{t_i}(T_1, \ldots, T_{j-1}) = \lambda_{c_i} = \frac{\lambda}{\binom{K}{d}}$ when $t_i \notin \{T_1, \ldots, T_{j-1}\}$. Similarly, we have that $\lambda_{\mathcal{U}(\{T_1, \ldots, T_i\})} = \frac{i\lambda}{\binom{K}{d}}$ and Condition \ref{cdn:assignment} is trivially satisfied.

To characterise the departure rates for the COC model, we assume that $\eta(j)=1$ for all $j \in \mathbb{N}$. Recall that $s_j(T_1, \ldots, T_i)$ can be interpreted as the departure rate of the customer that holds token $T_j$. This customer is the oldest of its class, and furthermore, among all the oldest customers within their class, the $j-1$-st oldest overall. The token mechanism dictates that customers holding tokens $T_1, \ldots, T_{j-1}$ in principle get priority over $T_j$ in receiving service. This holds true in the COC-model, since the customer holding $T_j$ will have a non-zero departure rate only when there exist servers that do not actually serve the customers holding tokens $T_1, \ldots, T_{j-1}$.  This occurs when the server associated with $T_j$ was not among the $d$ servers selected by the $j-1$ customers which arrived earlier. It is evident that the departure rate of a customer holding a token equals $\mu$ (the service rate obtained from a single server) times the number of servers that are working on copies of this customer. To summarise, when $F_j(T_1, \ldots, T_i)$ refers to the number of servers that are able to serve copies of at least one of the customers holding $T_1, \ldots, T_j$, $1\le j \le i$, we have that $s_j(T_1, \ldots, T_i) = \mu(F_j(T_1, \ldots, T_i) - F_{j-1}(T_1, \ldots, T_i))$. Note that, by nature of the function $F_j(\cdot)$, it is straightforward that $F_j(T_1, \ldots, T_i) = F_j(T_1, \ldots, T_j)$ and that $F_j(T_1, \ldots, T_j) = F_j(\bar{T}_1, \ldots, \bar{T}_j)$ for any permutation $(\bar{T}_1, \ldots, \bar{T}_j)$ of $(T_1, \ldots, T_j)$. As a consequence, $k(T_1, \ldots, T_i) = \mu F_i(T_1, \ldots, T_i)$ and the order-independent condition holds. Finally, we have that $G_{c_1, \ldots, c_i}(T_1, n_1, \ldots, T_i, n_i) = \ind{\bigcap_{j=1}^i \{T_i = t_i\}}$.

\paragraph{Performance analysis}\mbox{ }\\
By substituting of the above chosen model parameters, \eqref{thm:statDist} provides the stationary distribution of the COC model. For any state $\st = (T_1, n_1, \ldots, T_i, n_i)\in\stsp$ ,
\begin{equation*}
\pi(\st) = \pi((0)) \frac{1}{i!} \prod_{j=1}^i \left( \frac{j\lambda}{\mu\binom{K}{d}F_j(T_1, \ldots, T_i)}\right)^{n_j+1},
\end{equation*}
where $\pi((0))$ is a normalising constant. This stationary distribution is also given in \cite[Proposition 7]{Ayesta18}. However, unlike \cite{Ayesta18} or \cite{Gardner17}, we now also give transforms for the stationary number of customers in the system, as well as their sojourn time. More precisely, Corollary \ref{cor:jointPGFnumberCustPresentGeneral} now implies for $z \in \{\bar{c}\in \CC:|\bar{c}|<1\}$ that
\begin{equation*}
\E{z^M} = \sum_{i=0}^{\binom{K}{d}} \sum_{(T_1,\ldots, T_i)\in\mathcal{T}^{i}} \frac{(\lambda z)^i\pi((0))}{\prod_{j=1}^i \mu\binom{K}{d}F_j(T_1, \ldots, T_i)-j\lambda z}.
\end{equation*}
Furthermore, in the COC model it holds that customers of the same type depart the system in the order they arrive. Therefore, applying Theorem \ref{thm:sojournTime} (or more particularly, \eqref{eq:LSTScSimple}) and exploiting symmetry, we have for any $s \in \{\bar{c}\in\CC: \Re(\bar{c})>0\}$ that
\begin{align*}
\E{e^{-sS}} = \E{e^{-sS_{c_1}}} &= \sum_{i=0}^{\binom{K}{d}} \sum_{(T_1, \ldots, T_i)\in\mathcal{T}^{i}}\frac{\la-s\binom{K}{d}\ind{t_1\in\{T_1, \ldots, T_i\}}}{\la} \times\\
&\qquad\qquad\qquad\times\frac{\lambda^i\pi((0))}{\prod_{j=1}^i \left(\mu\binom{K}{d}F_j(T_1, \ldots, T_i)-j \lambda+s\binom{K}{d}\ind{t_1\in\{T_1, \ldots, T_j\}}\right)}.
\end{align*}
This concludes the performance analysis of the redundancy-$d$ COC model.


\subsection{Matching models}\label{sec:matching}
The last application that we discuss is the matching model. As discussed in \cite{AdanRighterWeiss1, AdanRighterWeiss2}, parallel FCFS matching models consist of independent arrival streams of several types of customers (the set of customer types is given by $\mathcal{T}$) and independent arrival streams of several types of servers. Each customer type is compatible to several types of servers, and the objective is to match customers with compatible servers. Customers wait in a queue until they are matched with a server. When a server arrives, it scans the queue of customers and matches with the longest waiting customer of a compatible customer type, after which both the customer and the server leave the system. In case a server finds no such customer, it departs immediately on its own. As mentioned in \cite{AdanRighterWeiss1, AdanRighterWeiss2}, these matching models have applications in many areas such as manufacturing, call centers and housing.

We apply our framework to a matching model where times between two arrivals of type-$c_i$ customers are independently and exponentially($\lambda_{c_i}$) distributed, $c_i \in \mathcal{C}$. Furthermore, arrivals of servers of any particular type form an inhomogeneous Poisson ($A(n)$) process. The varying rate $A(n)$ is modulated by the number of customers waiting in the queue: when there are $n$ customers in the queue, servers arrive at rate $A(n)$. It makes sense to assume that $A(n)$ is increasing in $n$, but this assumption is not needed for the analysis that follows. To the best of our knowledge, the case of a varying server arrival rate has not been considered before. We derive the stationary distribution and consider the distributions of the number of customers waiting for a match as well as the time that customers spend waiting.

\subsubsection{Choice of model parameters}
As it turns out, interpreting the matching model as a token-based central queue is analogous to the interpretation of a COC model as a token-based central queue. That is, if there are $K$ customer types, then the token set is given by $\mathcal{T} = \{t_1, \ldots, t_K\}$, where token $t_i$ corresponds to customer type $c_i$. When an arriving customer finds no customer of its type already waiting upon arrival, it claims the token corresponding to its type. If not, it is forced to wait until earlier-arrived customers have been matched in order to claim the token. When a token is claimed, it can be considered for a match by arriving servers. By the one-to-one correspondence of tokens to customer types, it is easy to see that $\lambda_{c_i}(T_1, \ldots, T_{j-1}) = \lambda_{t_i}$ when $t_i \notin \{T_1, \ldots, T_{j-1}\}$. In other words, the activation rate of a token is not dependent on the tokens already activated, given that $t_i$ itself is available. Due to this insensitivity, Condition \ref{cdn:assignment} holds true. We moreover have that $\lambda_{\mathcal{U}(\{T_1, \ldots, T_i\})} = \sum_{j=1}^i \lambda_{T_j}$. 

Interpreting the arrival of a server as a completion of service of a customer holding a token, we select the following departure rates. This time, we let $F_j(T_1, \ldots, T_i) = F_j(T_1, \ldots, T_j)$, $1 \le j \le i$, be the number of server types that are compatible to any of the customer types corresponding to the tokens $T_1, \ldots, T_j$. It is clear that $\eta(j) = A(j)$ for $j \in \mathbb{N}$, while $s_j(T_1, \ldots, T_i) = F_j(T_1, \ldots, T_i)-F_{j-1}(T_1, \ldots, T_i)$ for $1 \le j \le i$. As a consequence, $k(T_1, \ldots, T_i) = F_i(T_1, \ldots, T_i)$. It is straightforward that Condition \ref{cdn:OI} holds under these parameter settings and that $G_{c_1, \ldots, c_i}(T_1, n_1, \ldots, T_i, n_i) = \ind{\bigcap_{j=1}^i \{T_i = t_i\}}$.

\subsubsection{Performance analysis}
Since Conditions \ref{cdn:assignment} and \ref{cdn:OI} are satisfied, Theorem \ref{thm:statDist} confirms the finding of \cite{AdanEtAl12} that matching models generally allow for product-form stationary distributions. In particular, it follows from Theorem \ref{thm:statDist} that for any $\st = (T_1, n_1, \ldots, T_i, n_i)\in\stsp$,
\begin{equation*}
\pi(\st) = \pi((0)) \prod_{j=1}^i \frac{\lambda_{T_j}}{A(j)\sum_{k=1}^j \lambda_{T_k}}\left(\frac{\sum_{k=1}^j \lambda_{T_k}}{F_j(T_1, \ldots, T_i)}\right)^{n_j+1}\prod_{j=1}^{\sum_{k=1}^i n_k} \frac{1}{A(i+j)},
\end{equation*}
with $\pi((0))$ acting as a normalising constant. The stationary number of customers waiting to be matched (which due to the different notion now corresponds with $M$), is characterised by
\begin{align*}
\E{z^M} &= \sum_{i=0}^K \pi((0)) z^i \sum_{(T_1, \ldots, T_i)\in\mathcal{T}^{i}} \prod_{j=1}^i \frac{\lambda_{T_j}}{A(j)F_j(T_1, \ldots, T_i)}\times \\
&\qquad\qquad\qquad\times\sum_{\{n_1, \ldots, n_i\}\in\NN_0^i} \prod_{j=1}^{\sum_{k=1}^i n_k} \frac1{A(i+j)}\prod_{j=1}^i \left(\frac{\sum_{k=1}^j \lambda_{T_j}z}{F_j(T_1, \ldots, T_i)}\right)^{n_j},
\end{align*}
where $z \in \{\bar{c}\in\CC:|\bar{c}|<1\}$, due to Corollary \ref{cor:jointPGFnumberCustPresentGeneral}. Likewise, Equation \eqref{eq:LSTScSimple} leads, for any $s\in\{\bar{c}\in\CC:\Re(\bar(c))>0\}$ and any $l \in \{1, \ldots, K\}$ to
\begin{align*}
\E{e^{-sS_{c_l}}} &= \sum_{i=0}^{K} \pi((0)) \sum_{(T_1, \ldots, T_i)\in\mathcal{T}^{i}}\frac{\la_{t_l}-s\ind{t_l \in \{T_1, \ldots, T_i\}}}{\la_{t_l}}\prod_{j=1}^i \frac{\lambda_{T_j}}{A(j)F_j(T_1, \ldots, T_i)} \times \\
&\qquad\qquad\qquad\times \sum_{\{n_1, \ldots, n_i\}\in\NN_0^i} \prod_{j=1}^{\sum_{k=1}^i n_k} \frac 1{A(i+j)}\prod_{j=1}^i \left(\frac{\sum_{k=1}^j \lambda_{T_k}-s\ind{t_l\in\{T_1, \ldots, T_i\}}}{F_j(T_1, \ldots, T_i)}\right)^{n_j}.
\end{align*}

\begin{remark}
As mentioned, the results in this section are very similar to those for the redundancy-$d$ COC model in Section \ref{sec:coc}. This is not surprising, given the fact that in \cite{AdanRighterWeiss1, AdanRighterWeiss2} it is shown that matching models and redundancy models reveal a big similarity. In particular, for a matching model of the sort mentioned in this section, a redundancy model with similar parameter settings can be formulated so that the number of customers in the system is sample path equivalent for both systems (cf.\ \cite[Theorem 3.1]{AdanRighterWeiss2}). The reason why the results in Section \ref{sec:coc} and this section do not completely coincide is that the matching model we presented is more general. In Section \ref{sec:coc}, a redundancy-$d$ setting was assumed, while in this section, we took departure rates dependent on the number of customers in the system through the function $A(\cdot)$. Therefore, the results in Section \ref{sec:coc} can be recovered by choosing $K = \binom{K}{d}$, $\lambda_{T_j} = \frac{\la}{\binom{k}{d}}$ and $A(n)=1$ for all $n \in \NN$. 
\end{remark}

\section*{Acknowledgements}
The research of U. Ayesta, T. Bodas and I.M. Verloop was partially supported by the French ”Agence Nationale de la Recherche (ANR)” through the project ANR-15-CE25-0004 (ANR JCJC RACON). The research of J.L. Dorsman was funded by the NWO Gravitation project NETWORKS, grant number 024.002.003.

\bibliographystyle{plain}
\bibliography{bib}

\appendix
\section{Completion of proof of Theorem \ref{thm:statDist}}\label{app:thm:statDist}
\begin{proof}
We complete the proof of Theorem \ref{thm:statDist}. More particularly, we show below that \eqref{eq:statDist} satisfies \eqref{eq:balanceEquationArrival}, \eqref{eq:balanceEquationDepartureRelease} and \eqref{eq:balanceEquationDepartureShift}.

\subparagraph{The stationary distribution \eqref{eq:statDist} satisfies \eqref{eq:balanceEquationArrival}.}
Note that in case \eqref{eq:statDist} holds, for $n_i>0$, the right-hand side of \eqref{eq:balanceEquationArrival} can be rewritten as follows. We have that
\begingroup
\allowdisplaybreaks
\begin{align*}
&\la_{\mathcal{U}(\{T_1, \ldots, T_i\})}\pi((T_1, n_1, \ldots, T_i, n_i-1))\\
&\;=\; \la_{\mathcal{U}(\{T_1, \ldots, T_i\})}\pi((0))\frac{\Pi_{\lambda}(\{{T}_1, \ldots, {T}_i\})}{\Pi_{k}({T}_1,\ldots,{T}_i)}\left(\prod_{j=1}^i{\alpha_j}^{n_j}\right)\frac{1}{\alpha_i}\left(\prod_{j=1}^{\phi(\st)}\frac{1}{\eta(j)}\right)\eta(\phi(\st)) \\
&\;=\;\eta(\phi(\st))k(T_1, \ldots, T_i)\pi((0))\frac{\Pi_{\lambda}(\{{T}_1, \ldots, {T}_i\})}{\Pi_{k}({T}_1,\ldots,{T}_i)}\prod_{j=1}^i{\alpha_j}^{n_j} \prod_{j=1}^{\phi(\st)}\frac{1}{\eta(j)} \\
&\;=\;\mu(\st)\pi(\st),
\end{align*}
\endgroup
where the last equality follows by virtue of \eqref{eq:mu}. As this is the left-hand side of \eqref{eq:balanceEquationArrival}, we conclude that \eqref{eq:statDist} satisfies this equation for $n_i>0$. Similarly, for $n_i=0$, the right-hand side of \eqref{eq:balanceEquationArrival} can be rewritten as
\begingroup
\allowdisplaybreaks
\begin{align*}
&\la_{T_i}(\{T_1, \ldots, T_{i-1})\pi((T_1, n_1, \ldots, T_{i-1}, n_{i-1}))\\
&\;=\; \la_{T_i}(\{T_1, \ldots, T_{i-1})\pi((0))\frac{\Pi_{\lambda}(\{{T}_1, \ldots, {T}_{i-1}\})}{\Pi_{k}({T}_{1},\ldots,{T}_{i-1})}\prod_{j=1}^{i-1}{\alpha_j}^{n_j} \left(\prod_{j=1}^{\phi(\st)}\frac{1}{\eta(j)}\right)\eta(\phi(\st)) \\
&\;=\; \eta(\phi(\st))k(T_1, \ldots, T_i)\pi((0))\frac{\Pi_{\lambda}(\{{T}_1, \ldots, {T}_i\})}{\Pi_{k}({T}_1,\ldots,{T}_i)}\left(\prod_{j=1}^{i-1}{\alpha_j}^{n_j}\right)\alpha_i^0 \prod_{j=1}^{\phi(\st)}\frac{1}{\eta(j)} \\
&\;=\;\mu(\st)\pi(\st). \\
\end{align*}
\endgroup
Again, the last equality follows from \eqref{eq:mu}, and we have shown that \eqref{eq:statDist} satisfies \eqref{eq:balanceEquationArrival}.

\subparagraph{The stationary distribution \eqref{eq:statDist} satisfies \eqref{eq:balanceEquationDepartureRelease}.}
We follow the same strategy as before. That is, we substitute \eqref{eq:statDist} into the right-hand side of \eqref{eq:balanceEquationDepartureRelease}.
We obtain
\begingroup
\allowdisplaybreaks
\begin{align*}
&\sum_{k=0}^i\sum_{n=0}^{n_k} \mu_T(\mbr{k}{n}{\st}{T})\rprob{k}{n}{\st}{T}\pi(\mbr{k}{n}{\st}{T}) \\
&\;=\; \sum_{k=0}^i\sum_{n=0}^{n_k} \eta(\phi(\st)+1)(k(T_1, \ldots,T_k, T)-k(T_1, \ldots, T_k)) \beta_k(T)^n\left(\prod_{j=k+1}^i\beta_{j}(T)^{n_{j}}\right)\times \\
&\qquad\qquad\times\pi((0))  \frac{\Pi_{\lambda}(\{{T}_1, \ldots, {T}_k, T, T_{k+1}, \ldots, T_i\})}{\Pi_k(T_1, \ldots, T_{k}, T, T_{k+1}, \ldots, T_i)}\left(\prod_{j=1}^{k-1}\alpha_j^{n_j}\right)\alpha_k^{n_k-n}\left(\frac{\lambda_{\mathcal{U}(\{T_1, \ldots, T_k, T\})}}{k(T_1, \ldots, T_k, T)}\right)^n\times \\
&\qquad\qquad\times\left(\prod_{j=k+1}^i\left(\frac{\lambda_{\mathcal{U}(\{T_1, \ldots, T_j,T\})}}{k(T_1, \ldots, T_k, T, T_{k+1}, \ldots, T_j)}\right)^{n_j}\right) \prod_{j=1}^{\phi(\st)+1}\frac{1}{\eta(j)} \\
&\;=\; \pi(0)\left(\prod_{j=1}^{\phi(\st)}\frac{1}{\eta(j)}\right)\sum_{k=0}^i(k(T_1, \ldots, T_k, T)-k(T_1, \ldots, T_k))\frac{\Pi_{\lambda}(\{{T}_1, \ldots, \ldots, T_i, T\})}{\Pi_k(T_1, \ldots, T_k, T, T_{k+1}, \ldots, T_i)}\left(\prod_{j=1}^{k-1}\alpha_j^{n_j}\right) \times\\
&\qquad\qquad\times \left(\prod_{j=k+1}^i\left(\frac{\beta_{j}(T)\lambda_{\mathcal{U}(\{T_1,  \ldots, T_j,T\})}}{k(T_1, \ldots, T_k, T, T_{k+1}, \ldots, T_j)}\right)^{n_j}\right) \sum_{n=0}^{n_k}\alpha^{n_k-n}\left(\frac{\beta_k(T)\lambda_{\mathcal{U}(\{T_1, \ldots, T_k, T\})}}{k(T_1, \ldots, T_k, T)}\right)^n \\
&\;=\;\la_T(\{T_1, \ldots, T_i\})\pi(0)\frac{\Pi_\lambda(\{T_1, \ldots, T_i)\}}{\Pi_k(T_1, \ldots, T_k)}\left(\prod_{j=1}^i\alpha_j^{n_j}\right)\left(\prod_{j=1}^{\phi(\st)}\frac{1}{\eta(j)}\right) \times \\
&\qquad\qquad\times\sum_{k=0}^i\frac{k(T_1, \ldots, T_k, T)-k(T_1, \ldots, T_k)}{k(T_1, \ldots, T_k, T)}\prod_{j=k+1}^i\left(\frac{\lambda_{\mathcal{U}}(\{T_1, \ldots, T_j\})}{\alpha_j k(T_1, \ldots, T_j, T)}\right)^{n_j} \times \\
&\qquad\qquad\times\prod_{j=k+1}^i\frac{k(T_1, \ldots, T_j)}{k(T_1, \ldots, T_j, T)}\sum_{n=0}^{n_k} \left(\frac{\lambda_{\mathcal{U}(\{T_1, \ldots, T_k\})}}{\alpha_k k(T_1, \ldots, T_k, T)}\right)^n \\
&\;=\; \la_T(\{T_1, \ldots, T_i\})\pi(\st) \sum_{k=0}^i\frac{k(T_1, \ldots, T_k, T)-k(T_1, \ldots, T_k)}{k(T_1, \ldots, T_k, T)} \times\\
&\qquad\qquad\times\prod_{j=k+1}^i\left(\frac{k(T_1, \ldots, T_j)}{k(T_1, \ldots, T_j, T)}\right)^{n_j+1}\sum_{n=0}^{n_k} \left(\frac{k(T_1, \ldots, T_k)}{k(T_1, \ldots, T_k, T)}\right)^n.
\end{align*}
\endgroup
Here, the second equality follows from Condition \ref{cdn:assignment}. The third equality follows from Condition \ref{cdn:OI} and the fact that $\Pi_k(T_1, \ldots, T_j, T, T_{k+1}, \ldots, T_i) = k(T_1, \ldots, T_i, T)\Pi_k(T_1, \ldots, T_i)\prod_{j=k+1}^i\frac{k(T_1, \ldots, T_j, T)}{k(T_1, \ldots, T_j)}$, which is straightforwardly verifiable. It is left to show that the last line of the display equals $\la_T(\{T_1, \ldots, T_i\})\pi(\st)$. We can indeed show that the outer sum in the last line of the display indeed equals one. We note that for $k=0$, we have that $n_0=0$, since there can be no customers in the system that can claim $T$. Then, we have  
\begingroup
\allowdisplaybreaks
\begin{align}
&\sum_{k=0}^i\frac{k(T_1, \ldots, T_k T)-k(T_1, \ldots, T_k)}{k(T_1, \ldots, T_k, T)}\prod_{j=k+1}^i\left(\frac{k(T_1, \ldots, T_j)}{k(T_1, \ldots, T_j, T)}\right)^{n_j+1}\sum_{n=0}^{n_k} \left(\frac{k(T_1, \ldots, T_k)}{k(T_1, \ldots, T_k, T)}\right)^n \notag\\
&\;=\;\sum_{k=0}^i\left(1-\frac{k(T_1, \ldots, T_k)}{k(T_1, \ldots, T_k, T)}\right)\prod_{j=k+1}^i\left(\frac{k(T_1, \ldots, T_j)}{k(1, \ldots, T_j, T)}\right)^{n_j+1}\sum_{n=0}^{n_k} \left(\frac{k(T_1, \ldots, T_k)}{k(T_1 \ldots, T_k, T)}\right)^n \notag\\
&\;=\;\sum_{k=0}^i\left(1-\left(\frac{k(T_1, \ldots, T_k)}{k(T_1, \ldots, T_k, T)}\right)^{n_k+1}\right)\prod_{j=k+1}^i\left(\frac{k(T_1, \ldots, T_j)}{k(T_1, \ldots, T_j, T)}\right)^{n_j+1} \notag\\
&\;=\;\sum_{k=0}^i \left(\prod_{j=k+1}^i\left(\frac{k(T_1, \ldots, T_j)}{k(T_1, \ldots, T_j, T)}\right)^{n_j+1}-\prod_{j=k}^i\left(\frac{k(T_1, \ldots, T_j)}{k(T_1, \ldots, T_j, T)}\right)^{n_j+1}\right) \notag\\
&\;=\; \prod_{j=i+1}^i\left(\frac{k(T_1, \ldots, T_j)}{k(T_1, \ldots, T_j, T)}\right)^{n_j+1}-\prod_{j=0}^i\left(\frac{k(T_1, \ldots, T_j)}{k(T_1, \ldots, T_j, T)}\right)^{n_j+1} \notag\\
&\;=\;1. \label{eq:kTermEqualsOne}
\end{align}
\endgroup
The last equality follows since the first product equals one by definition, and the second product equals zero by the fact that $k(\emptyset) = 0$. Therefore, we have shown that \eqref{eq:statDist} satisfies \eqref{eq:balanceEquationDepartureRelease}.

\subparagraph{The stationary distribution \eqref{eq:statDist} satisfies \eqref{eq:balanceEquationDepartureShift}.}
Again, we follow a similar line of reasoning as before. Manipulating the right-hand side of \eqref{eq:balanceEquationDepartureShift}, we have
\begingroup
\allowdisplaybreaks
\begin{align*}
&\sum_{j=1}^i \sum_{k=0}^{j-1}\sum_{n=0}^{n_k} \mu_{T_j}(\mbs{k}{n}{\st}{T_j})\sprob{k}{n,}{s}{T_j}\pi(\mbs{k}{n}{\st}{T_j}))\\
&\;=\; \sum_{j=1}^i \sum_{k=0}^{j-1} \sum_{n=0}^{n_k} \eta(\phi(\st)+1) \left(k(T_1, \ldots, T_k, j)-k(T_1, \ldots, T_k)\right) \beta_k(T_j)^n\left(\prod_{l=k+1}^{j-1} \beta_l(T_j)^{n_l}\right)(1-\beta_{j-1}(T_j)) \times \\
&\qquad\qquad\times \pi((0)) \frac{\Pi_\lambda(\{T_1, \ldots, T_k, T_j, T_{k+1}, \ldots, T_{j-1}, T_{j+1}, \ldots, T_i\})}{\Pi_k(T_1, \ldots, T_k, T_j, T_{k+1}, \ldots, T_{j-1}, T_{j+1}, \ldots, T_i)} \left(\prod_{l=1}^{k-1}\alpha_l^{n_l}\right)\alpha_k^{n_k-n}\times\\
&\qquad\qquad\times\left(\frac{\lambda_{\mathcal{U}(\{T_1, \ldots, T_k, T_j\})}}{k(T_1, \ldots, T_j, T_k)}\right)^n \left(\prod_{l=k+1}^{j-1}\left(\frac{\lambda_{\mathcal{U}(\{T_1, \ldots, T_l, T_j\})}}{k(T_1, \ldots, T_k, T_j, T_{k+1}, \ldots, T_l)}\right)^{n_l}\right)\times\\
&\qquad\qquad\times\left(\frac{\lambda_{\mathcal{U}(\{T_1, \ldots, T_j\})}}{k(T_1, \ldots, T_k, T_j, T_{k+1}, \ldots, T_{j-1})}\right)^{n_j+1}\times \\
&\qquad\qquad\times\left(\prod_{l=j+1}^i \left(\frac{\lambda_{\mathcal{U}(\{T_1, \ldots, T_l\})}}{k(T_1, \ldots, T_k, T_j, T_{k+1}, \ldots, T_{j-1}, T_{j+1}, \ldots, T_l)}\right)^{n_l}\right)\prod_{l=1}^{\phi(\st)+1}\frac{1}{\eta(l)}. \\
&\;=\; \pi((0)) \frac{\prodlambda}{\prodk} \left(\prod_{j=1}^{\phi(\st)} \frac{1}{\eta(j)}\right) \sum_{j=1}^i \left(1- \frac{\lambda_{\mathcal{U}(\{T_1, \ldots, T_{j-1}\})}}{\lambda_{\mathcal{U}(\{T_1, \ldots, T_j\})}}\right)\sum_{k=0}^{j-1} \left(k(T_1, \ldots, T_k, T_j)-k(T_1, \ldots, T_k)\right) \times \\
&\qquad\qquad\times\left(\frac{\prod_{l=k+1}^{j} k(T_1, \ldots, T_l)}{\prod_{l=k}^{j-1} k(T_1, \ldots, T_l, T_j)}\right)\left(\prod_{l=1}^{k-1}\alpha_l^{n_l}\right)\left(\prod_{l=k+1}^{j-1}\left(\frac{\lambda_{\mathcal{U}(\{T_1, \ldots, T_l, T_j\})}}{k( T_1, \ldots, T_l, T_j)}\right)^{n_l}\right)
\left(\frac{\lambda_{\mathcal{U}(\{T_1, \ldots, T_j\})}}{k(T_{1}, \ldots, T_j)}\right)^{n_j+1} \times \\
&\qquad\qquad\times\left(\prod_{l=j+1}^i\left(\frac{\lambda_{\mathcal{U}(\{T_1, \ldots, T_l\})}}{k(T_1, \ldots, T_l)}\right)^{n_l}\right)\sum_{n=0}^{n_k}\alpha_k^{n_k-n}\left(\frac{\lambda_{\mathcal{U}(\{T_1, \ldots, T_k})}{k(T_1, \ldots, T_k, T_j)}\right)^n \\
&\;=\;\pi((0)) \frac{\prodlambda}{\prodk} \prod_{j=1}^i \alpha_j^{n_j}\left(\prod_{j=1}^{\phi(\st)}\frac{1}{\eta(j)}\right)\sum_{j=1}^i \left(\lambda_{\mathcal{U}(\{T_1, \ldots, T_j\})}- \lambda_{\mathcal{U}(\{T_1, \ldots, T_{j-1}\})}\right) \times \\
&\qquad\qquad\times\sum_{k=0}^{j-1}\left(1-\frac{k(T_1, \ldots, T_k)}{k(T_1, \ldots, T_k, T_j)}\right)\left(\prod_{l=k+1}^{j-1}\frac{k(T_1, \ldots, T_l)}{k(T_1, \ldots, T_l, T_j)}\right)\left(\prod_{l=k+1}^{j-1}\left(\frac{\lambda_{\mathcal{U}(\{T_1, \ldots, T_l, T_j\})}}{\alpha_j k(T_1, \ldots, T_l, T_j)}\right)^{n_l}\right) \times \\
&\qquad\qquad\times\sum_{n=0}^{n_k}\left(\frac{\lambda_{\mathcal{U}(\{T_1, \ldots, T_k})}{\alpha_k k(T_1, \ldots, T_k, T_j)}\right)^n \\
&\;=\;\pi(\st) \sum_{j=1}^i \left(\lambda_{\mathcal{U}(\{T_1, \ldots, T_j\})}- \lambda_{\mathcal{U}(\{T_1, \ldots, T_{j-1}\})}\right) \times \\
&\qquad\qquad\times \sum_{k=0}^{j-1}\left(1-\frac{k(T_1, \ldots, T_k)}{k(T_1, \ldots, T_k, T_j)}\right)\left(\prod_{l=k+1}^{j-1}\left(\frac{k( T_1, \ldots, T_l)}{k(T_1, \ldots, T_l, T_j)}\right)^{n_l+1}\right)\sum_{n=0}^{n_k}\left(\frac{ k( T_1, \ldots, T_k)}{ k(T_1, \ldots, T_k, T_j)}\right)^n \\
&\;=\; \lambda_{\mathcal{U}(\{T_1, \ldots, T_i\})}\pi(\st),
\end{align*}
\endgroup
which is the left-hand side of \eqref{eq:balanceEquationDepartureShift}. In the second equality, we used Condition \ref{cdn:assignment} and Condition \ref{cdn:OI}. The final equality follows by invoking \eqref{eq:kTermEqualsOne} with $i=j-1$ and the fact that $\sum_{j=1}^i \left(\lambda_{\mathcal{U}(\{T_1, \ldots, T_j\})}- \lambda_{\mathcal{U}(\{T_1, \ldots, T_{j-1}\})}\right) = \lambda_{\mathcal{U}(\{T_1, \ldots, T_i\})}$. As we have rewritten the right-hand side of \eqref{eq:balanceEquationDepartureShift} into its left-hand side, we conclude that \eqref{eq:statDist} satisfies \eqref{eq:balanceEquationDepartureShift}, which completes the proof.
\end{proof}

\section{Proof of Theorem \ref{thm:jointPGFNumberCustWaitingClass}}\label{app:thm:jointPGFNumberCustWaitingClass}
\begin{proof} To prove the theorem, we first consider $N_{j}^{(c)}$, which we recall to be the number of type-$c$ customers among $N_j$, which represents the number of customers in the central queue between those that have claimed $T_j$ and $T_{j+1}$. More particularly, we first focus on (the joint PGF of) the stationary distribution of $\{N_{j}^{(c)}: j \in \{1, \ldots, K\}, c \in \mathcal{C}\}$   using Theorem \ref{thm:statDist} as a starting point. Then, the expression for $\{N^{(c)}: c \in \mathcal{C}\}$ as given in Theorem \ref{thm:jointPGFNumberCustWaitingClass} will follow almost immediately.

We use Theorem \ref{thm:statDist} as a starting point. From this theorem, we gather that the stationary distribution of the model at hand satisfies 
\begin{equation}\label{eq:stdist}
\pi(T_1, n_1, \ldots, T_i, n_i) = \pi((0))\frac{\prodlambda}{\prodk} \prod_{j=1}^i\alpha_j^{n_j} \prod_{j=1}^i\frac 1{\eta(j)} \prod_{j=1}^{\sum_{k=1}^i n_k} \frac 1{\eta(i+j)}.
\end{equation}
By the dynamics of the arrival process, we next note that $N_j^{(c)}$ is binomially distributed with parameters $N_j$ and $\theta_{c,j} := \frac{\lambda_c\ind{c \in \mathcal{U}(T_1, \ldots, T_j)}}{\lambda_{\mathcal{U}({T_1, \ldots, T_j}})}$. The indicator function in this expression reflects the fact that in order for $N_{j}^{(c)}$ to be positive, any token in the set $\mathcal{T}\backslash\{T_1, \ldots, T_j\}$ must reject class-$c$ jobs. More generally, the set $\{N_j^{(c)}: c \in \mathcal{C}\}$ is multinomially distributed with population size parameter $N_j$ and probability parameters $\{\theta_{c,j}:c \in \mathcal{C}\}$. We also observe that, given the values of $N_1, N_2, \ldots$, the sets $\{N_1^{(c)}:c \in \mathcal{C}\}, \{N_2^{(c)}:c \in \mathcal{C}\}, \ldots$ are independent, so that
\begin{equation}\label{eq:NjcGivenNj}
\Pro{\bigcap_{j\in\{1, \ldots, i\}, c \in \mathcal{C}} \{N_j^{(c)} = n_j^{(c)}\} \mid \bigcap_{j=1}^i\{N_j = n_j\}} = \prod_{j=1}^i \frac{n_j!}{\prod_{c \in \mathcal{C}}n_j^{(c)}!}\prod_{c \in \mathcal{C}} \theta_c^{n_{c,j}}.
\end{equation}
Using \eqref{eq:NjcGivenNj} and applying Newton's binomium, respectively, immediately leads to the following joint PGF. For $z_{c,j} \in \{\bar{c}\in \CC : |\bar{c}|\le 1\}$,
\begin{align}
&\E{\prod_{c \in \mathcal{C}} \prod_{j=1}^i z_{c,j}^{N_j^{(c)}} \mid \st = (T_1, n_1, \ldots, T_i, n_i)} \notag \\
&= \sum_{\{n_j^{(c)}: c \in \mathcal C\}:\sum_{c \in \mathcal{C}} n_j^{(c)} = n_j}\frac{n_j!}{\prod_{c \in \mathcal{C}}n_j^{(c)}!}\prod_{c \in \mathcal{C}}(\theta_{c,j}z_{c,j})^{n_{c,j}} =  \prod_{j=1}^i\left(\sum_{c \in \mathcal{C}} \theta_{c,j}z_{c,j}\right)^{n_j}. \label{eq:NjcGivenState}
\end{align}
Unconditioning using \eqref{eq:stdist} now leads to
\begin{align}
&\E{\prod_{c \in \mathcal{C}}\prod_{j=1}^K z_{c,j}^{N_j^{(c)}}} \notag \\
&=\sum_{i=0}^K \sum_{(T_1, \ldots, T_i)\in \mathcal{T}^{i}} \sum_{(n_1, \ldots, n_i) \in \NN_0^i} \pi((T_1, n_1, \ldots, T_i, n_i)) \E{\prod_{c \in \mathcal{C}}\prod_{j=1}^K z_{c,j}^{N_j^{(c)}}\mid \st = (T_1, n_1, \ldots, T_i, n_i)} \notag \\
&=\sum_{i=0}^K \sum_{(T_1, \ldots, T_i)\in \mathcal{T}^{i}} \pi((0))\frac{\prodlambda}{\prodk} \prod_{j=1}^i\frac 1{\eta(j)}\sum_{\{n_1, \ldots, n_i\} \in \NN_0^i}\prod_{j=1}^{\sum_{k=1}^i n_k} \frac 1{\eta(i+j)} \prod_{j=1}^i (\al_j\sum_{c \in \mathcal{C}}\theta_{c,j}z_{c,j})^{n_j}.  
\end{align}
Finally, Equation \eqref{eq:jointPGFNcUnconditional} follows by combining this expression with
\begin{equation*}
\E{\prod_{c \in \mathcal{C}}z_c^{N^{(c)}}} = \E{\prod_{c \in \mathcal{C}}z_c^{\sum_{j=1}^k N_j^{(c)}}} = \E{\prod_{c \in \mathcal{C}}\prod_{j=1}^k z_c^{N_j^{(c)}}}.
\end{equation*}
\end{proof}

\section{Proof of Theorem \ref{thm:jointPGFnumberCustPresentClass}}\label{app:thm:jointPGFnumberCustPresentClass}
\begin{proof} The proof hinges on the notion that, if there are at least $j$ tokens activated, either $M_j^{(c)} = N_j^{(c)}+1$ if token $T_j$ is claimed by a type-$c$ customer, or $M_j^{(c)} = N_j^{(c)}$ otherwise. This leads to
\begin{align*}
&\E{\prod_{c\in\mathcal{C}}\prod_{j=1}^i z_{c,j}^{M_j^{(c)}}} \notag \\
&= \sum_{i=0}^K \sum_{(T_1, \ldots, T_i)\in \mathcal{T}^{i}} \sum_{(n_1, \ldots, n_i) \in \NN_0^i} \pi((T_1, n_1, \ldots, T_i, n_i))\E{\prod_{c \in \mathcal{C}} \prod_{j=1}^i z_{c,j}^{N_j^{(c)}} \mid \st = (T_1, n_1, \ldots, T_i, n_i)}\times\\
&\qquad\qquad\times\sum_{(c_1, \ldots, c_i)\in \mathcal{C}^i}G_{c_1, \ldots, c_i}(T_1, n_1, \ldots, T_i, n_i) \prod_{j=1}^iz_{c,j}
\end{align*}
The theorem now follows by substitution of \eqref{eq:stdist} and \eqref{eq:NjcGivenState} into this expression and realising that $\E{\prod_{c \in \mathcal{C}}z_c^{M^{(c)}}} = \E{\prod_{c \in \mathcal{C}}\prod_{j=1}^k z_c^{M_j^{(c)}}}$. 
\end{proof}

\section{Expressions for performance measures when $\eta(\cdot) = 1$}\label{app:eta1}
It follows by substitution and subsequent simplification of \eqref{eq:jointPGFNcUnconditional}, \eqref{eq:jointPGFUnconditionalSimple}, \eqref{eq:jointPGFMUnconditional}, \eqref{eq:LSTWc} and \eqref{eq:LSTSc} that, when $\eta(j) = 1$ for all $j \in \NN$,
\begin{equation*}
\E{\prod_{c \in \mathcal{C}}z_c^{N^{(c)}}}=\sum_{i=0}^K \sum_{(T_1, \ldots, T_i)\in \mathcal{T}^{i}} \pi((0))\frac{\prodlambda}{\prodk}  \prod_{j=1}^i \frac1{1-\al_j\sum_{c \in \mathcal{C}}\theta_{c,j}z_{c}},
\end{equation*}
\begin{equation*}
\E{z^N}=\sum_{i=0}^K \sum_{(T_1, \ldots, T_i)\in \mathcal{T}^{i}} \pi((0))\frac{\prodlambda}{\prodk}  \prod_{j=1}^i \frac1{1-\al_jz}, 
\end{equation*}
\begin{equation*}
\E{z^M}=\sum_{i=0}^K \sum_{(T_1, \ldots, T_i)\in \mathcal{T}^{i}} \pi((0))\frac{\prodlambda}{\prodk} z^i\prod_{j=1}^i \frac1{1-\al_jz} 
\end{equation*}
and
\begin{equation*}
\E{e^{-sW_c}} = \sum_{i=0}^K \sum_{(T_1, \ldots, T_i)\in \mathcal{T}^{i}} \pi((0))\prod_{j=1}^i \frac{\lambda_{T_j}(T_1, \ldots, T_{j-1})}{k(T_1, \ldots, T_j)-\lambda_{\mathcal{U}(\{T_1, \ldots, T_j\})}+s\ind{c\in \mathcal \mathcal{U}(\{T_1, \ldots, T_j\})}}, 
\end{equation*}
with $\pi((0))$ as given in \eqref{eq:piZeroSimple}.

\section{Proof of Theorem \ref{lem:keep}}\label{app:lem:keep}
\begin{proof}
We first assume that $(1)$ holds, that is, we are given a model that fits in the OI queue framework. 
In the remainder of this proof, we will use the notion of indistinguishable tokens and token labels as introduced in Section \ref{sec:L}. We define the following token sets $\mathcal{T}_c$. Each token set of customer type~$c$ consists of an infinite number of indistinguishable tokens with label $c$. Thus, every customer type has its dedicated token label.  Then, the state $\st^{(L)} = (L_1,   \ldots, L_n)$ gives exactly the same information as the state $\stOI=(c_1, \ldots, c_n)$, hence both state descriptors are equivalent. When setting  $\mu_{L_j}(\st^{(L)}) = {\mu}_{j}^{(OI)}(\stOI)$, the token-based central queue describes exactly the same model as the OI queue. 

What is left to show is that the token-based central queue satisfies Condition~\ref{cdn:assignment} and Condition~\ref{cdn:OI}. 
Condition~\ref{cdn:OI} follows directly, since $\mu_{L_j}(\st^{(L)})= {\mu}_{j}^{(OI)}(\stOI)$ and  $\mu_j^{(OI)}$ satisfies Condition~\ref{cond:X}. 
Since each customer type $c$ has its own dedicated set of indistinguishable tokens with label $c$, we have that $\lambda_{c}(\{L_1, \ldots, L_i \})=\lambda_c$. Therefore, $\prod_{j=1}^{i} \lambda_{L_j}(\{L_1, \ldots L_{j-1} \})=\prod_{j=1}^{i} \lambda_{L_j}$. This expression is independent of the permutation of the $L_j$'s, and since tokens that bear the same label are indistinguishable, Condition~\ref{cdn:assignment} is satisfied. We have hence proved that $(1)\to (2)$.


We now assume that $(2)$ of Theorem \ref{lem:keep} holds, that is, we are given a token-based central queue  where each token set consists of indistinguishable tokens. 
From Lemma~\ref{lem:tau}, it follows that to a given state $\stOI=(c_1,\ldots, c_n)$ (describing the type of each customer), there corresponds a unique state $\st^{(L)} = (L_1, n_1,\ldots, L_i, n_i)$, given by $\tau(\stOI)$. 
We also define the function $\mathcal{\tilde \tau}({\stOI})$ that gives the  unique  activated tokens $(L_1, \ldots, L_i)$ as a function of $\stOI$. 
To prove that (1) of Theorem \ref{lem:keep} holds, that is, the model fits the OI queue framework, we will (i) define functions $\eta^{(OI)}(\cdot)$ and $s_j^{(OI)}(\cdot)$, (ii) show that these functions give rise to an OI queue and  (iii) show that the departure rates under the token-based central queue and the OI queue are sample-path wise equal. 

(i) Since $\phi(\mathcal{\tau}(\stOI)) = n$, we  define $$\eta^{(OI)}(n):=\eta(\phi(\mathcal{\tau}(\stOI))).$$ 
Let $h(j,\stOI):=\sum_{c \in \mathcal{C}} \min( \sum_{l=1}^j \mathbf{1}_{(c_l=c)}  , |\mathcal{T}_c|)$ denote the number of active customers among the first~$j$ customers (for ease of exposition, we assume that any two tokens from any two token sets $\mathcal{T}_{c_a}$ and $\mathcal{T}_{c_b}$, $c_a, c_b\in\mathcal{C}$, are not indistinguishable). 
When in state $\stOI=(c_1,\ldots, c_n)$ and if $\sum_{i=1}^{j} \mathbf{1}_{(c_j=c_i)} \leq |\mathcal{T}_{c_j}|$, then the $j$-th customer is the $h(j,\stOI)$-th customer that has a token. We therefore define
$$
s_j^{(OI)}(\stOI):=
\begin{cases}
s_{{h(j,\stOI)}}(\mathcal{\tilde \tau}(\stOI)) \mbox{ if } \sum_{i=1}^{j} \mathbf{1}_{(c_j=c_i)} \leq |\mathcal{T}_{c_j}|,\\
0, \mbox{  otherwise.}
\end{cases}
$$

(ii) Since Condition~\ref{cdn:OI} is satisfied, it is immediate that $\mu(\stOI)$ satisfies Condition~\ref{cond:X} and hence gives rise to an OI queue. 

(iii) Consider the $j$-th customer. We now show that its departure rate in both systems is the same, which  concludes the proof.  If $\sum_{i=1}^{j} \mathbf{1}_{(c_j=c_i)} \leq |\mathcal{T}_{c(j)}|$, then the $j$-th customer  has a token and is the  $h(j, \stOI)$-th active customer in the queue. 
Its departure rate in the OI queue is 
$\mu_j^{(OI)}(\stOI)= \eta^{(OI)}(n) s_j^{(OI)}(\stOI)  =  \eta(\phi(\mathcal{\tau}(\stOI))) s_{{h(j,\stOI)}}(\mathcal{\tilde \tau}(\stOI))= \mu_{L_{h(j,\stOI)}}(\mathcal{\tau}(\stOI))$, which equals its departure rate in the token-based central queue.  
 If the $j$-th customer is not active, then its departure rate  in the OI queue is $\mu_j^{(OI)}(\stOI)=0$, which equals its departure rate in  the token-based central queue.
\end{proof}

\section{Proof of Corollary~\ref{cor}}
\label{app:cor}
\begin{proof} 
 In the proof of Theorem~\ref{lem:keep} $(1)\to(2)$ it was shown that the OI queue can be seen as a token-based central queue where a token set $\mathcal{T}_c$ of a customer type $c$ consists of infinitely many indistinguishable tokens with label $c$. 
Noting that  $\phi(\st)=n$, $n_j=0$ and $\lambda_{c}(\{L_1, \ldots, L_j\})=\lambda_c$,  from \eqref{eq:statdistL} we recover the product-form stationary distribution~\eqref{eq:OIstat}  for the OI state descriptor.
\end{proof}

\end{document}